\documentclass[10pt,reqno]{amsart}
\usepackage{amsfonts,amsmath,pdfsync,color,graphicx,calligra,mathrsfs,tikz,tikz-cd,mathtools}
\usepackage{amsthm,amssymb}
\usepackage{marginnote}
\usetikzlibrary{arrows}

\usepackage[english]{babel}			
\usepackage[latin1]{inputenc}
\usepackage[T1]{fontenc}

\usepackage[vmargin=3cm,hmargin=3cm]{geometry}

\newcommand{\LLL}{\mathbb{L}}
\newcommand{\N}{\mathbb{N}}

\newcommand{\Z}{\mathbb{Z}}

\newcommand{\aaa}{\mathfrak{a}}
\newcommand{\mmm}{\mathfrak{m}}

\newcommand{\nnn}{\mathfrak{n}}

\newcommand{\qqq}{\mathfrak{q}}

\newcommand{\pr}{\operatorname{pr}}
\newcommand{\coker}{\operatorname{coker}}

\newcommand{\Der}{\operatorname{Der}}

\newcommand{\Ext}{\operatorname{Ext}}
\newcommand{\EExt}{\mathscr{E}\text{\kern -3pt {\calligra\large xt}}\ }

\newcommand{\Hom}{\operatorname{Hom}}
\newcommand{\HHom}{\mathscr{H}\text{\kern -3pt {\calligra\Large om}}\, }
\newcommand{\id}{\operatorname{id}}
\newcommand{\im}{\operatorname{Im}}

\newcommand{\Spec}{\operatorname{Spec}}
\newcommand{\Tor}{\operatorname{Tor}}
\newcommand{\Set}{\operatorname{Set}}
\newcommand{\gp}{\text{gp}}
\newcommand{\llog}{\text{log}}
\newcommand{\arrowr}{\arrow{r}}
\newcommand{\arrowd}{\arrow{d}}
\newcommand{\limdir}{\varinjlim}

\newcommand{\Lim}[1]{\raisebox{0.7ex}{\scalebox{0.8}{$\displaystyle\limdir_{i}\;$}}}

\newcommand*{\longhookrightarrow}{\ensuremath{\lhook\joinrel\relbar\joinrel\rightarrow}}

\newtheorem{theorem}{Theorem}[section]
\newtheorem*{theorem*}{Theorem}
\newtheorem{proposition}[theorem]{Proposition}
\newtheorem*{proposition*}{Proposition}
\newtheorem{lemma}[theorem]{Lemma}

\newtheorem{corollary}[theorem]{Corollary}
				
\theoremstyle{definition}
\newtheorem{definition}[theorem]{Definition}
\newtheorem{example}[theorem]{Example}				
\newtheorem{examples}[theorem]{Examples}				
	
\newtheoremstyle{mystyle}{}{}{}{}{\bfseries \itshape}{.}{ }{}
\theoremstyle{mystyle}
\newtheorem{remark}[theorem]{Remark}

\title[Homological characterization of regularity in LAG]{Homological characterization of regularity\\ in Logarithmic Algebraic Geometry}

\author[J. Conde--Lago]{Jes\'us Conde--Lago $^{(\star)}$}
\date{}
\address{Departamento de Matem\'aticas, Facultade de Matem\'aticas, Universidade de Santiago de Compostela, E-15782 Santiago de Compostela, Spain}
\email{jesus.conde@usc.es}

\author[J. Majadas]{Javier Majadas $^{(\star)}$}
\address{Departamento de Matem\'aticas, Facultad de Matem\'aticas, Universidad de Santiago de Compostela, E-15782 Santiago de Compostela, Spain}
\email{j.majadas@usc.es}

\thanks{$^{(\star)}$ This work was partially supported by Agencia Estatal de Investigaci\'{o}n (Spain), grant MTM2016-79661-P (European FEDER support included, UE). J. Conde-Lago was also supported by a scholarship, Xunta de Galicia (European Social Fund support included, UE)}

\date{\today}

\begin{document}

\maketitle

\vspace{-5ex}
\begin{abstract}
	We characterize K. Kato's log regularity in terms of vanishing of (co)homology of the logarithmic cotangent complex.
\end{abstract}
\vspace{5ex}
	
\section{Introduction}

In 1967, M. Andr\'e \cite{An-MS} and D. Quillen \cite{Quillen-MIT} introduce homology and cohomology modules \linebreak $H_n(A,B,W)$, $H^n(A,B,W)$ for an (always commutative) $A$-algebra $B$ and a $B$-module $W$, giving vanishing criteria for regularity and complete intersection in \cite[corollaries~10.8 and 10.12]{Quillen-MIT} and\linebreak \cite[\S27-\S28]{An-MS}:
\begin{theorem*}
	Let $(A,k)$ be a noetherian local ring. Then
	\begin{enumerate}
		\item[(i)] $(A,k)$ is regular if and only if $H_2(A,k,k)=0$.
		\item[(ii)]  $(A,k)$ is complete intersection if and only if $H_3(A,k,k)=0$.
	\end{enumerate}
\end{theorem*}

Formal smoothness was also characterized in terms of vanishing in \cite[Proposition 5.3]{Quillen-MIT} and \cite[Proposition~16.17]{An-1974} (see also \cite[0$_{\text{IV}}$.19.4.4]{EGA4}):
\begin{theorem*}
	Let $B$ be a noetherian $A$-algebra and $J$ an ideal of $B$. Then $B$ is formally smooth over $A$ for the $J$-adic topology if and only if $H^1(A,B,W)=0$ for all $B/J$-modules $W$.
\end{theorem*}

Since these (co)homology modules have a good behaviour, these vanishing criteria were immediately used 
to obtain new results (and easier proofs of already known results) on smoothness, regularity and complete intersection (see e.g. \cite{An-1974}, the references in this book to papers by Brezaleanu and Radu, \cite{An-LLF}, etc.).

The concept of log regularity (and log smoothness) in logarithmic algebraic geometry was introduced by Kato in \cite[Definition 2.1, Theorem 6.1]{Ka-TS}: let $(X,\mathcal{M})$ be a ``not very bad'' locally noetherian log scheme. Then $(X,\mathcal{M})$ is regular at a point $x$ if $A/I$ is a local regular ring and $\dim A=\dim A/I + \dim M^\gp/A^*$ where $A=\mathcal{O}_{X,x}$, $M=\mathcal{M}_x$, the unit group $A^*$ is identified canonically to a subgroup of $M$, $I$ is the ideal of $A$ generated by the image of $M-A^*$, and $M^\gp$ is the group completion of the monoid $M$.

We have also in this logarithmic context a definition of ``log'' Andr\'e-Quillen (co)homology (in fact, we have two definitions, one by O. Gabber and the other by M. Olsson; we will use Gabber's one. Both constructions are developed in \cite{Ol}; see also Sagave, Sch\"{u}rg and Vezzosi's paper \cite{SSV}). Moreover the characterization of log smoothness by vanishing of this cohomology is done in \cite{Ol}.

So the next step is to give vanishing criteria for log regularity. This is done in Theorem~\ref{6.01} and more particularly in Theorem~\ref{6.09} (see Section~\ref{LR} for notation in the (pre)logarithmic setting): let $(A,M)$ be a noetherian local prelog ring with $M$ integral. Without loss of generality, we can assume that $M^\gp$ is a free abelian group (see Theorem~\ref{6.03}, Lemma~\ref{6.06}); then
\begin{theorem*}
	$(A,M)$ is a log regular if and only if $H_2((A,M),(k,M/\mmm_M),k)=0$ (where $k$ is the residue field of $A$ and $\mmm_M\subset M$ is the ideal of non-units).
\end{theorem*}

The main ingredient of the proof is a long exact sequence (Theorem~\ref{4.03}) relating the (Gabber's) log Andr\'e-Quillen (co)homology of a homomorphism of prelog rings $(A,M)\to(B,N)$ with the (usual) Andr\'e-Quillen (co)homology of the ring homomorphisms $\Z[M]\to\Z[N]$ and $A\to B$. This exact sequence is also very useful later for the applications, since once we have vanishing results in the log and non-log setting, we can relate log smoothness, log regularity and log complete intersection properties with the (non-log) corresponding properties of the underlying rings. In fact, we are able to reprove at the end of Section~\ref{S} results that are more or less already known, but usually our proofs are simpler and our hypotheses sometimes weaker. For example, we avoid the usual hypotheses of finitely generated and saturated monoids. We even avoid integrity in some of the main results. Also the morphisms in Section 5 are not necessarity of finite type in most cases.
Some of the applications of our main theorem at the end of Section~\ref{R} also use the long exact sequence of Theorem~\ref{4.03}.

The content of the paper is as follows. In Section~\ref{LR} we give definitions and introduce the notation for the logarithmic setting. We continue with definitions, notation and known results on derivations and differentials in this setting in Section~\ref{DD}. Though there is nothing new in this section, we give proofs since in a few points we were not able to find exact references.

In Section~\ref{CC} we introduce the Gabber's log Andr\'e-Quillen (co)homology and we obtain some results. This theory was already developed in \cite[\S8]{Ol}, but we include a complete exposition for several reasons. The first one is that a few of the properties we need are not included in \cite{Ol} (e.g. Theorem~\ref{4.03}, Proposition~\ref{4.04}, Proposition~\ref{4.05} and Proposition~\ref{4.10}) and others appear in \cite{Ol} but with stronger hypothesis (e.g. Proposition~\ref{4.08}). The second one is that since we only need a local definition (i.e. for rings instead of sheaves of rings), we can take any simplicial resolution (instead of a standard resolution) since they are all homotopically equivalent, making some proofs easier. And third, we want to promote the exact sequence of Theorem~\ref{4.03} relating log Andr\'e-Quillen (co)homology with non-log Andr\'e-Quillen (co)homology (for instance Proposition~\ref{4.14} is exactly \cite[Theorem 8.20]{Ol}, but we include here a different proof in order to show the usefulness of Theorem~\ref{4.03}; propositions~\ref{4.04} and~\ref{4.08} are also deduced from Theorem~\ref{4.03}).

In Section 5 we characterize log formal smoothness by the vanishing of cohomology. First we characterize the elements of $H^1$ as extensions. This is made in \cite{Ol} and outlined in \cite{Il-GL}, but we include here a proof since in \cite{Ol} (see algo \cite{SSV}) it is assumed integrity and in \cite{Il-GL} there is no explicit reference to cohomology. Characterization of log smoothness by the vanishing of $H^1$ is then deduced (\cite{Ol}; see also \cite{SSV}), but again we include proofs in order to take into account topologies
: theorems \ref{5.14} and \ref{5.15} describe the situation. Finally we give a few examples where this characterization is useful. Topologies are important when the morphisms are not of finite type: for example, in Theorem~\ref{6.14} we extend Kato's result \cite[8.3]{Ka-TS} comparing log smoothness and log regularity to morphisms not necessarily of finite type.

Section~\ref{R} contains the main result of the paper (the characterization of regularity) and some applications. We work in a more general context defining and characterizing \emph{log regular ideals}.

\section{Logarithmic rings}\label{LR}
We will fix some notation, see e.g. \cite{Og} for proofs. 

All rings and monoids will be commutative and all categorical constructions with rings (or monoids) will be within the category of commutative rings (or monoids). In addition, we will usually use multiplicative notation for monoids.

\begin{definition}\label{2.01}
	A \emph{congruence} $R$ in a monoid $M$ is an equivalence relation on $M$ compatible with multiplication, that is
	\vspace{-0.05cm}
	\[ (x,x'),(y,y')\in R \ \Rightarrow \ (xy,x'y')\in R . \]
	The quotient set $M/R$ is a monoid with the multiplication induced by the one in $M$.
	
	An \emph{ideal} $I$ of a monoid $M$ is a nonempty subset $I\subset M$ such that $ma\in I$ for all $m\in M$, $a\in I$. If $I\neq M$ is an ideal of $M$, then $R:=(I\times I)\cup \Delta(M)$ is a congruence on $M$, where $\Delta(M)=\{(x,x)\in M\times M\}$, and the quotient monoid will be denoted by $M/I$. If $M$ is a monoid and $A$ is a ring, $M^*$ and $A^*$ will denote their subgroups of units. The quotient monoid given by the congruence $R:=\{(x,ux) \mid x\in M, u\in M^*\}$ will be denoted $M/M^*$. If $A$ is a ring, $I$ is an ideal of $A$, $M$ is a monoid and $\alpha\colon M\to (A,\cdot)$ is a homomorphism of monoids, then $\alpha^{-1}(I)$ is an ideal of $M$. We have a homomorphism of monoids $M/\alpha^{-1}(I)\to(A/I,\cdot)$.
	
	If $M$ is a monoid and $S$ a submonoid, the localization $S^{-1}M$ will be the quotient monoid of $M\times S$ by the congruence defined as follows: $(x,s)\sim(y,t)$ in $M\times S$ if there exists $u\in S$ such that $syu=xtu$. An example is the group completion of $M$ (i.e. when $S=M$), denoted $M^\gp$. A monoid is \emph{integral} if the canonical homomorphism $M\to M^\gp$ is injective.
	
	If $M$ is a monoid and $R$ is a ring, $R[M]$ will denote the usual monoid $R$-algebra. If $M\to N$, $M\to P$ are monoid homomorphisms, we denote their pushout by $N\oplus_MP$.
	
	A \emph{prelog ring} $(A,M,\alpha)$ (or simply $(A,M)$ or $A$ if there is no confusion) consists of a commutative ring $A$, a commutative monoid $M$ and a multiplicative homomorphism of monoids $\alpha\colon M\to A$. In general, we will denote this structural map by $\alpha$ for any prelog ring (if it is necessary, to avoid confusion, we will use $\alpha_A$ or $\alpha_M$). A homomorphism of prelog rings \[f=(f^\sharp,f^\flat)\colon  (A,M,\alpha_A)\longrightarrow (B,N,\alpha_B)\] is a homomorphism of rings $f^\sharp\colon A\to B$ together with a homomorphism of monoids $f^\flat\colon M\to N$ such that $\alpha_B f^\flat =f^\sharp \alpha_A$. A \emph{log ring} is a prelog ring $(A,M,\alpha)$ such that the homomorphism $\alpha^{-1}(A^*)\to A^*$ induced by $\alpha$ (where $A^*$ is the group of units of $A$) is an isomorphism. If $(A,M,\alpha)$ is a prelog ring, define $M^\llog$ as the pushout
	\[
	\begin{tikzcd}[row sep=4em, column sep=4em]
	\alpha^{-1}(M) \arrowr{\alpha}\arrowd & A^* \arrowd\\
	M\arrowr & M^\llog
	\end{tikzcd}
	\]
	and let $\alpha^\llog\colon M^\llog\to A$ be the monoid homomorphism induced by the homomorphisms $\alpha\colon M \to A$ and $A^*\hookrightarrow A$. Then $(A,M,\alpha)^\llog:=(A,M^\llog,\alpha^\llog)$ is a log ring. The functor $(\ \ )^\llog$ from the category of prelog rings to the category of log rings is left adjoint to the forgetful functor.
	
	If $(A,M,\alpha)$ is a prelog ring and $f^\sharp\colon A\to B$ is a ring homomorphism, we have a homomorphism of prelog rings \[f=(f^\sharp,f^{b})\colon (A,M,\alpha)\longrightarrow (B,f^{*}(M),\alpha^*)\] where $(B,f^{*}(M),\alpha^*):=(B,M,f\circ\alpha)^\llog$. A homomorphism of log rings \[f\colon (A,M)\longrightarrow (B,N)\] is called \emph{strict} if the canonical map $f^*(M)\to N$ is a isomorphism.
	
	A noetherian local prelog ring $((A,\mathfrak{m},k),M)$ is a noetherian local ring $(A,\mathfrak{m},k)$, a monoid $M$ such that the ring $\mathbb{Z}[M]$ is noetherian (i.e. $M$ is finitely generated \cite[Theorem 20.7]{G}) and a local homomorphism of monoids $\alpha\colon M\to (A,\cdot)$ (that is, $\alpha(\mathfrak{m}_M)\subset \mathfrak{m}$, where $\mathfrak{m}_M=M-M^*$).
	
	We say that $(A,M)\to(B,N)$ is a homomorphism (essentially) of finite type of noetherian prelog rings if $M$, $N$ are finitely generated monoids, $A$ is a noetherian ring, and $A\to B$ is a homomorphism (essentially) of finite type. We have then that $\Z[M]\to \Z[N]$ is a homomorphism of finite type of noetherian rings, and $M^\gp$, $N^\gp$ are $\Z$-modules of finite type.
\end{definition}

\section{Derivations and differentials}\label{DD}
\begin{proposition}\label{3.01}
	Let $g\colon  L\to N$ be a surjective homomorphism of monoids and define $J:=\ker(\Z[L]\to\Z[N])$, $W:=\ker(L^\gp\to N^\gp)$ (note that $(-)^\gp$ and $\Z[-]$ preserve surjectivity). There exists a natural homomorphism of $\Z[N]$-modules
	\[\nu_g \colon  \  J/ J^2 \longrightarrow \Z[N]\otimes_\Z W\]
	defined by
	\[\nu_g\big(\overline{\sum_{i\in I}\lambda_i l_i}\big):=\sum_{i\in I} \lambda_i g(l_i)\otimes l_i \quad \in \Z[N]\otimes_\Z L^\gp , \]
	where $\lambda_i\in\Z$ and $l_i\in L$ for all $i\in I$, and we are denoting also by $l_i$ the image of $l_i$ in $L^\gp$.	
\end{proposition}
\begin{proof}
	We define $\widetilde{\nu_g}\colon  J\to \Z[N]\otimes_\Z L^\gp$ as $\widetilde{\nu_g}(\sum_{i\in I}\lambda_i l_i):=\sum_{i\in I} \lambda_i g(l_i)\otimes l_i$. First, we will show that $\widetilde{\nu_g}(x)\in \Z[N]\otimes_\Z W$ if $x=\sum_{i\in I}\lambda_i l_i\in J$. For each $n\in N$, we consider $I_n=\{i\in I \mid g(l_i)=n\}$, so that $\sum_{i\in I_n}\lambda_i=0$ for all $n\in N$. We have
	\[ \sum_{i\in I}\lambda_ig(l_i)\otimes l_i =\sum_{i\in I} g(l_i)\otimes l_i^{\lambda_i}=\sum_{n\in N} \sum_{i\in I_n} n\otimes l_i^{\lambda_i} = \sum_{n\in N} \big(n\otimes \prod_{i\in I_n} l_i^{\lambda_i}\big) \]
	and $\prod_{i\in I_n} l_i^{\lambda_i}\in W$ since if $g$ denotes also the map $L^\gp\to N^\gp$, then
	\[g\big(\prod_{i\in I_n}{l_i}^{\lambda_i}\big)=\prod_{i\in I_n} g(l_i)^{\lambda_i}=\prod_{i\in I_n} n^{\lambda_i}=n^{\sum_{i\in I_n} \lambda_i}=n^0=1 .\]
	The map $\widetilde{\nu_g}$ is additive, so in order to see that it induces a homomorphism
	\[\nu_g \colon \  J/ J^2 \longrightarrow \Z[N]\otimes_\Z W\]
	we have to show that $\widetilde{\nu_g}(x)=0$ for $x\in  J^2$. The elements of $ J$ are of the form
	\[ \sum_{i\in I} (\lambda_i l_{i,1}-\lambda_i l_{i,2}) \]
	with $\lambda_i\in\Z$, $l_{i,1},l_{i,2}\in L$ and  $g(l_{i,1})=g(l_{i,2})$ for all $i\in I$. Therefore let $\lambda_1,\lambda_2\in\Z$ and $l_{i,j}\in L$ for $i,j\in\{1,2\}$ such that $g(l_{1,1})=g(l_{1,2})=\colon n_1$ and $g(l_{2,1})=g(l_{2,2})=\colon n_2$. We obtain
	\begin{align*}
	\widetilde{\nu_g}\big((\lambda_1l_{1,1}-\lambda_1l_{1,2})(\lambda_2l_{2,1}-\lambda_2l_{2,2})\big)&=\lambda_1\lambda_2g(l_{1,1})g(l_{2,1})\otimes l_{1,1}l_{2,1} - \lambda_1\lambda_2g(l_{1,2})g(l_{2,1})\otimes l_{1,2}l_{2,1}\\
	&\quad - \lambda_1\lambda_2g(l_{1,1})g(l_{2,2})\otimes l_{1,1}l_{2,2} + \lambda_1\lambda_2g(l_{1,2})g(l_{2,2})\otimes l_{1,2}l_{2,2} \\
	&= \lambda_1\lambda_2n_1n_2\otimes l_{1,1}l_{2,1}-\lambda_1\lambda_2n_1n_2\otimes l_{1,2}l_{2,1} \\
	&\quad - \lambda_1\lambda_2n_1n_2\otimes l_{1,1}l_{2,2} + \lambda_1\lambda_2n_1n_2\otimes l_{1,2}l_{2,2} \\
	&= \lambda_1\lambda_2n_1n_2\otimes l_{1,1}l_{2,1}l_{1,2}^{-1}l_{2,1}^{-1}l_{1,1}^{-1}l_{2,2}^{-1}l_{1,2}l_{2,2}\\
	&= \lambda_1\lambda_2n_1n_2\otimes 1=0 .
	\end{align*}
	
	We now see that $\nu_g$ is a homomorphism of $\Z[N]$-modules. We already know that it is an additive homomorphism. Let $m\in N$ be and $l\in L$ such that $g(l)=m$. With the above definition for $I_n$,
	\begin{align*}
	\nu_g\big(m\cdot \overline{\sum_{i\in I_n}\lambda_i l_i}\big) &= \widetilde{\nu_g}(\sum_{i\in I_n}\lambda_i l l_i)=\sum_{i\in I_n} \lambda_i g(ll_i)\otimes ll_i = \sum_{i\in I_n} \lambda_i mn\otimes ll_i =mn\otimes\prod_{i\in I_n} l^{\lambda_i} l_i^{\lambda_i}\\
	&=mn\otimes \big(l^{\sum_{i\in I_n}\lambda_i} \prod_{i\in I_n} l_i^{\lambda_i}\big) =mn\otimes \big(l^{0} \prod_{i\in I_n} l_i^{\lambda_i}\big) =mn\otimes\prod_{i\in I_n} l_i^{\lambda_i} = m\cdot \widetilde{\nu_g}\big(\sum_{i\in I_n}\lambda_i l_i\big) \\
	&= m\cdot \nu_g \big(\overline{\sum_{i\in I_n}\lambda_i l_i} \big) .
	\end{align*}
	
	Finally, naturality means that if
	\[
	\begin{tikzcd}[row sep=4em, column sep=4em]
	L_1 \arrowr{g_1}\arrowd{\gamma} & N_1 \arrowd{\rho} \\
	L_2 \arrowr{g_2} & N_2 
	\end{tikzcd}
	\]
	is a commutative square of monoids with $g_1,g_2$ surjective homomorphisms, and $ J_1,W_1, J_2,W_2$ are defined in the obvious way, then the diagram
	\[
	\begin{tikzcd}[row sep=4em, column sep=3em]
	J_1/ J_1^2 \arrowr{\nu_{g_1}}\arrowd{\widetilde{\gamma}} & \Z[N_1]\otimes_\Z W_1 \arrowd{\widetilde{\rho}} \\
	J_2/ J_2^2 \arrowr{\nu_{g_2}} & \Z[N_2]\otimes_\Z W_2
	\end{tikzcd}
	\]
	is commutative, and this fact is obvious.
\end{proof}

\begin{remark}\label{3.02}
	We have seen in the proof of Proposition~\ref{3.01} that the elements of $J$ are $\sum_{i\in I}(\lambda_il_{i,1}-\lambda_il_{i,2})$ with $\lambda_i\in\Z$ and $l_{i,1},l_{i,2}\in L$ verifying $g(l_{i,1})=g(l_{i,2})$ for all $i\in I$, and the image of one such element by $\widetilde{\nu_g}$ is $\sum_{i\in I}\lambda_ig(l_{i,2})\otimes l_{i,1}l_{i,2}^{-1}$.
\end{remark}

\begin{example}\label{3.03}
	Let $h\colon  M\to N$ be a homomorphism of monoids and consider the induced surjective homomorphism $g\colon  N\oplus_MN \to N$. We have $\Z[N\oplus_MN]=\Z[N]\otimes_{\Z[M]}\Z[N]$ and then $J:=\ker(\Z[N\oplus_MN]\to \Z[N])$ is the kernel of the multiplication $\Z[N]\otimes_{\Z[M]}\Z[N]\to \Z[N]$. Therefore we have an isomorphism of $\Z[N]$-modules $J/ J^2=\Omega_{\Z[N]|\Z[M]}$, where $\Omega_{\Z[N]|\Z[M]}$ is the usual module of differentials.
	
	On the other hand, we have a exact sequence
	\[
	\begin{tikzcd}[column sep=2em,row sep=0ex]
	1\arrowr & M^\gp \arrowr& N^\gp \arrowr &N^\gp\oplus_{M^\gp}N^\gp \arrowr & N^\gp \arrowr & 1 \\
	& & a\rar[maps to] & (a,a^{-1}) & \\
	& & & (a,b) \rar[maps to] & ab
	\end{tikzcd}
	\]
	so we obtain an isomorphism of abelian groups
	\begin{align*}
	W:=\ker((N\oplus_MN)^\gp\to N^\gp)&=\ker(N^\gp\oplus_{M^\gp}N^\gp \to N^\gp) \\
	&=\coker(M^\gp\to N^\gp).
	\end{align*}
	
	Then the composition
	\[ \Omega_{\Z[N]|\Z[M]}= J/ J^2\overset{\nu_g}{\longrightarrow}\Z[N]\otimes_\Z W=\Z[N]\otimes_\Z N^\gp/\im(M^\gp) \]
	is the homomorphism of $\Z[N]$-modules defined by
	\[ dn\longmapsto (n,1)-(1,n)\longmapsto n\otimes(n,1)-n\otimes(1,n)=n\otimes(n,1)(1,n)^{-1}=n\otimes(n,n^{-1})\longmapsto n\otimes\bar{n}, \]
	where $\bar{n}$ is the class of $n\in N$ in $N^\gp/\im(M^\gp)$.
\end{example}

\begin{definition}\label{3.04}
	Let $(f,g)\colon  (C,L)\to (B,N)$ a surjective homomorphism of prelog rings (i.e. $f\colon  C\to B$ and $g\colon L\to N$ are surjective). Let us consider $ J=\ker(\Z[L]\to\Z[N])$, $W=\ker(L^\gp\to N^\gp)$ and $I=\ker(C\to B)$. We define the \textit{conormal module} of $(C,L)\to(B,N)$ as the pushout $N_{(B,N)|(C,L)}$ of the diagram of $B$-modules
	\[
	\begin{tikzcd}[row sep=4em, column sep=5em]
	B\otimes_{\Z[N]} J/ J^2 \arrowr{\id_B\otimes_{\Z[N]}\nu_g}\arrowd{u} & B\otimes_\Z W \\
	I/I^2 & 
	\end{tikzcd}
	\]
	where $\nu_g$ is the homomorphism of Proposition~\ref{3.01} and $u$ is the homomorphism of $B$-modules defined by $u\big(1\otimes\overline{\sum_{i\in I}\lambda_il_i}\big):=\overline{\sum_{i\in I}\lambda_i \alpha(l_i)}$, where $\alpha\colon  L\to C$ is the structural homomorphism and $\lambda_i\in\Z$, $l_i\in L$ for all $i\in I$.
\end{definition}

\begin{remark}\label{3.05}
	The definition of \cite[Expos\'e II, Proposition 4.11]{Il-GL} coincides with this one.
\end{remark}

\begin{proposition}\label{3.06}
	Let
	\[
	\begin{tikzcd}[row sep=4em, column sep=3em]
	(C_1,L_1) \arrowr\arrowd & (B_1,N_1)\arrowd \\
	(C_2,L_2) \arrowr & (B_2,N_2)
	\end{tikzcd}
	\]
	be a commutative square of prelog rings where the horizontal homomorphisms are surjective. Then there exists a natural homomorphism
	\[B_2\otimes_{B_1}N_{(B_1,N_1)|(C_1|L_1)} \longrightarrow N_{(B_2,N_2)|(C_2|L_2)} .\]
\end{proposition}
\begin{proof}
	It is consequence of the naturality of the homomorphisms $\nu_g$ and $u$ in Definition~\ref{3.04}.
\end{proof}

\begin{example}\label{3.07}
	If $L\to N$ is an isomorphism, then $ J=0$ and $W=\{1\}$. Therefore the conormal module of $(C,L)\to(B,N)$ coincides with the (usual) conormal module $I/I^2$ of $C\to B$. This takes place, for example, when $L=N=\{1\}$. 
\end{example}

\begin{remark}\label{3.08}
	If $C$ is noetherian and $L^\gp$ is a finitely generated abelian group, then $N_{(B,N)|(C,L)}$ is a $B$-module of finite type. 
\end{remark}

\begin{definition}\label{3.09}
	Let $(A,M)\to (B,N)$ be a homomorphism of prelog rings. The \textit{module of diffe\-rentials $\Omega_{(B,N)|(A,M)}$ of $(A,M)\to (B,N)$} is defined as the $B$-module $N_{(B,N)|(C,L)}$ where $(C,L)=$\linebreak$(B\otimes_AB,N\oplus_MN)$. By the Example~\ref{3.03}, the square
	\[
	\begin{tikzcd}[row sep=4em, column sep=2em]
	B\otimes_{\Z[N]}\Omega_{\Z[N]|\Z[M]} \arrowr{v}\arrowd{u} & B\otimes_\Z N^\gp/\im(M^\gp)\arrowd \\
	\Omega_{B|A} \arrowr & \Omega_{(B,N)|(A,M)}
	\end{tikzcd}
	\]
	is cocartesian, where $v$ is the homomorphism of $B$-modules verifying $v(1\otimes dn)=n\otimes \bar{n}$ with $n\in N$ and $\bar{n}$ its class in $N^\gp/\im(M^\gp)$. If $A$ is noetherian, $B$ an $A$-algebra essentially of finite type and $N^\gp/\im(M^\gp)$ an abelian group of finite type, then $\Omega_{(B,N)|(A,M)}$ is a $B$-module of finite type. 
	
	By Proposition~\ref{3.06}, if	
	\[
	\begin{tikzcd}[row sep=4em, column sep=3em]
	(A_1,M_1) \arrowr\arrowd & (B_1,N_1)\arrowd \\
	(A_2,M_2) \arrowr & (B_2,N_2)
	\end{tikzcd}
	\]
	is a commutative square of homomorphisms of prelog rings, then there exists a natural homomorphism of $B_2$-modules
	\[B_2\otimes_{B_1}\Omega_{(B_1,N_1)|(A_1,M_1)}\longrightarrow\Omega_{(B_2,N_2)|(A_2,M_2)}.\]
\end{definition}

The following two results are well-known.

\begin{proposition}[Base Change]\label{3.10}
	We consider a diagram of log rings
	\[
	\begin{tikzcd}[row sep=4em, column sep=3em]
	(A_1,M_1) \arrowr\arrowd & (B_1,N_1) \\
	(A_2,M_2) &
	\end{tikzcd}
	\]
	and let $(B_2,N_2)$ be its pushout, i.e. $B_2=A_2\otimes_{A_1}B_1$, $N_2=M_2\oplus_{M_1}N_1$. Then the canonical homomorphism
	\[B_2\otimes_{B_1}\Omega_{(B_1,N_1)|(A_1,M_1)} \longrightarrow \Omega_{(B_2,N_2)|(A_2,M_2)}\]
	is an isomorphism.
\end{proposition}

\begin{proposition}\label{3.12}
	Let $(A,M)\to (B,N)\to (C,P)$ be homomorphisms of prelog rings. There exists a natural exact sequence of $C$-modules
	\[ C\otimes_B\Omega_{(B,N)|(A,M)} \longrightarrow \Omega_{(C,P)|(A,M)} \longrightarrow \Omega_{(C,P)|(B,N)} \longrightarrow 0 . \]
\end{proposition}

\begin{proposition}\label{3.13}
	Let $(A,M)\to (C,L)\to (B,N)$ be homomorphisms of prelog rings with $C\to B$ and $L\to N$ surjective. There exists a natural exact sequence of $B$-modules
	\[ N_{(B,N)|(C,L)} \longrightarrow B\otimes_C\Omega_{(C,L)|(A,M)} \longrightarrow \Omega_{(B,N)|(A,M)} \to 0 . \]
\end{proposition}
\begin{proof}
	It is deduced from the following exact sequences:
	\[ B\otimes_{\Z[N]} J/ J^2 \longrightarrow B\otimes_{Z[L]}\Omega_{\Z[L]|\Z[M]}\longrightarrow B\otimes_{Z[N]}\Omega_{\Z[N]|\Z[M]} \longrightarrow 0 , \]
	\[  I/ I^2 \longrightarrow B\otimes_{C}\Omega_{C|A}\longrightarrow \Omega_{B|A} \longrightarrow 0  \]
	and
	\[ B\otimes_{\Z}W \longrightarrow B\otimes_{Z}L^\gp/\im_{L^\gp}(M^\gp)\longrightarrow B\otimes_{Z}N^\gp/\im_{N^\gp}(M^\gp) \to 0 , \]
	where $ I=\ker(C\to B)$, $ J=\ker(\Z[L]\to\Z[N])$ and $W=\ker(L^\gp\to N^\gp)$. Note that the third sequence is exact because
	\[ W=\ker(L^\gp/\im_{L^\gp}(M^\gp)\to N^\gp/\im_{N^\gp}(M^\gp)) .\]
\end{proof}

\begin{examples}\label{3.14}
	\begin{enumerate}
		\item[(i)] Free prelog algebras. Let $(A,M)$ be a prelog ring, $X$ and $Y$ sets, $N=M\oplus \N^Y$ where $\N^Y$ is a direct sum of copies of the (additive) monoid of natural numbers indexed by $Y$, $B=A[X,Y]$ the polynomial $A$-algebra on $X\cup Y$, $N\to B$ the homomorphism which extends $M\to A$ in the obvious way. We will say that \emph{$(B,N)$ is a free $(A,M)$-prelog algebra}. Then $\Omega_{(B,N)|(A,M)}$ is a free $B$-module with basis $X\cup Y$.
		
		\item[(ii)] Let $(A,M)\to(B,N)$ a homomorphism essentially of finite type of noetherian prelog rings (Definition~\ref{2.01}). Then $\Omega_{(B,N)|(A,M)}$ is a $B$-module of finite type. It follows from the above example and Proposition~\ref{3.13}.	 
		
		\item[(iii)] Let $(A,M)\to (B,N)$ be a homomorphism of prelog rings with $M\to N$ surjective, then $\Omega_{(B,N)|(A,M)}=\Omega_{B|A}$. In fact, $N\oplus_MN\to N$ is an isomorphism and therefore the result follows from Example~\ref{3.07}.	
	\end{enumerate}
\end{examples}

\begin{definition}\label{3.15}
	Let $(A,M)\to (B,N)$ a homomorphism of prelog rings, $T$ a $B$-module. An $(A,M)$-\emph{derivation} of $(B,N)$ into $T$ is a pair $(D,\widetilde{D}$), where $D$ is an $A$-derivation $D\colon  B\to T$ and $\widetilde{D}$ is a homomorphism of groups
	\[\widetilde{D}\colon  N^\gp/\im(M^\gp)\longrightarrow (T,+)\]
	such that $D(n)=n\widetilde{D}(\bar{n})$ for all $n\in N$ (where $\bar{n}$ is its image in $N^\gp/\im(M^\gp)$ and we denote the image of $n\in N$ in $B$ again by $n$). We denote by
	\[\Der_{(A,M)}((B,N),T)\]
	the set of $(A,M)$-derivations of $(B,N)$ into $T$, which is a $B$-module with the induced structure by the structure of $B$-module of $T$.
\end{definition}

\begin{proposition}\label{3.16}
	With the above notation, there exists an isomorphism of $B$-modules
	\[ \Hom_B(\Omega_{(B,N)|(A,M)},T)=\Der_{(A,M)}((B,N),T) .\]
\end{proposition}
\begin{proof}
	Let $(D,\widetilde{D})\in \Der_{(A,M)}((B,N),T)$. From the diagram
	\[
	\begin{tikzpicture}[scale=2]
	\node (A) at (0,2.3) {$B\otimes_{\Z[N]}\Omega_{\Z[N]|\Z[M]}$};
	\node (B) at (2,2.3) {$B\otimes_{\Z}N^\gp/\im(M^\gp)$};
	\node (C) at (0,1) {$\Omega_{B|A}$};
	\node (D) at (2,1) {$\Omega_{(B,N)|(A,M)}$};
	\node (E) at (3.5,0) {$T$};
	\path[->,font=\scriptsize]
	(A) edge node[above]{$i$} (B)
	(A) edge node[left]{$j$} (C)
	(B) edge node[left]{$u$} (D)
	(C) edge node[above]{$v$} (D);
	\path[->,bend left=15,font=\scriptsize]
	(B) edge node[above]{$g$} (E);
	\path[->,bend right=15,font=\scriptsize]
	(C) edge node[above]{$f$} (E);				
	\path[->,dashed,font=\scriptsize]
	(D) edge node[above]{$h$} (E);
	\end{tikzpicture}
	\]
	where $f$ is the homomorphism of $B$-modules induced by $D\colon B\to T$ and $g$ is the homomorphism of $B$-modules induced by $\widetilde{D}\colon N^\gp / \im(M^\gp)\to T$, we obtain a homomorphism $h\colon  \Omega_{(B,N)|(A,M)}\to T$ making commutative the corresponding triangles.
	
	Now, consider $h\in \Hom_B(\Omega_{(B,N)|(A,M)},T)$. We define $D\in\Der_A(B,T)$ by the composition
	\[B\overset{d_{B|A}}{\longrightarrow}\Omega_{B|A}\overset{v}{\longrightarrow} \Omega_{(B,N)|(A,M)} \overset{h}{\longrightarrow} T\]
	and $\widetilde{D}$ by the composition
	\[N^\gp/\im(M^\gp)\longrightarrow B\otimes_\Z N^\gp/\im(M^\gp)\overset{u}{\longrightarrow} \Omega_{(B,N)|(A,M)} \overset{h}{\longrightarrow} T . \]
	Since $v$ and $h$ are homomorphism of $B$-modules, $D$ is a derivation. Similarly, $\widetilde{D}$ is a homomorphism of groups. We also have
	\[hui(1\otimes dn) = hu(n\otimes \bar{n}) = n\cdot hu(1\otimes\bar{n})=n\widetilde{D}(\bar{n})\]
	and $hvj(1\otimes dn)=D(n)$. Since $ui=vj$, we obtain $D(n)=n\widetilde{D}(n)$.
\end{proof}
	
\section{The cotangent complex}\label{CC}	
The forgetful functor $\mathcal{U}$ from the category of prelog rings to the category $\Set\times\Set$ sending $(A,M)$ to their subjacent sets $(\mathcal{U}A,\mathcal{U}M)$ has a left adjoint defined as
\[ F(X,Y):=(\Z[X,Y],\N^Y) ,\]
as in Example~\ref{3.14}.(i).

The category of simplicial prelog rings (i.e. simplicial objects in the category of prelog rings) has a model structure where (trivial) fibrations are those maps which go to (trivial) fibrations by the functor $s\mathcal{U}$ (the simplicial extension of the above functor $\mathcal{U}$). With this structure, (trivial) fibrations are those maps $(A_*,M_*)\to (B_*,N_*)$ such that $A_*\to B_*$ and $M_*\to N_*$ are (trivial) fibrations in the usual model structure of the categories of simplicial commutative rings and simplicial commutative monoids. If $(A,M)$ is a (constant) prelog ring, and $(A,M)\overset{\varphi}{\to}(B_*,N_*)$ is a homomorphism of simplicial prelog rings with the property that for each $n$, $(B_n,N_n)$ is a free $(A,M)$-prelog algebra, $B_n=A[X_n,Y_n]$, $N_n=M\oplus\N^{Y_n}$ for some sets $X_n,Y_n$, and all degeneracies $B_n=A[X_n,Y_n]\to B_{n+1}=A[X_{n+1},Y_{n+1}]$, $N_n\to N_{n+1}$ send the basis $X_n$ to $X_{n+1}$ and $Y_n$ to $Y_{n+1}$, then $\varphi$ is a cofibration called \emph{free cofibration} (arbitrary cofibrations are retracts of these ones). Details can be seen in \cite[\S4]{Quillen-HA} or more precisely for prelog rings in \cite{SSV} and in \cite[\S5]{Bh}.

Let $(A,M)\to (B,N)$ be a homomorphism of prelog rings. Choose a factorization $(A,M)\to(F,R)\to(B,N)$ as a cofibration followed by a trivial fibration (in the category of simplicial prelog rings). We have a well defined object up to homotopy in the category of simplicial $B$-modules
\[ \LLL_{(B,N)|(A,M)}:=\Omega_{(F,R)|(A,M)}\otimes_FB . \]

\begin{definition}\label{4.01}
	If $W$ is a $B$-module, we define for each $n\geq 0$,
	\begin{align*}
	&H_n((A,M),(B,N),W):=H_n(\LLL_{(B,N)|(A,M)}\otimes_BW) , \\
	&H^n((A,M),(B,N),W):=H^n(\Hom_B(\LLL_{(B,N)|(A,M)},W)) .
	\end{align*}
\end{definition}

\begin{example}\label{4.02}
	Let $A\to B$ be a homomorphism of rings, $(A,\{1\})\to(B,\{1\})$ the induced map on prelog rings. If $A\to F\to B$ is a factorization cofibration-trivial fibration in the category of simplicial rings, so is $(A,\{1\})\to(F,\{1\})\to (B,\{1\})$ in the category of simplicial prelog rings. Therefore $H_n((A,\{1\}),(B,\{1\}),W)=H_n(A,B,W)=H_n(\LLL_{B|A}\otimes_BW)$ are the usual Andr\'e-Quillen homology modules (see \cite{An-1974} and \cite{Quillen-MIT}) by Example~\ref{3.14}.(iii). Similarly for cohomology.
	
	More generally, if $A\to B$ is a ring homomorphism where $(A,M)$ is a prelog ring, then 
	\[ H_n((A,M),(B,M),W)=H_n(A,B,W) \]
	and similarly for cohomology. The argument is the same.
\end{example}

\begin{theorem}[Fundamental exact sequences]\label{4.03}
	Let $(A,M)\to(B,N)$ be a homomorphism of prelog rings and let $W$ be a $B$-module. There exists a natural distinguished triangle in the derived category of $B$-modules
	\[ \LLL_{\Z[N]|\Z[M]}\otimes_{\Z[M]}W \longrightarrow (\LLL_{B|A}\otimes_BW)\oplus(X_{N|M}\otimes_\Z W) \longrightarrow \LLL_{(B,N)|(A,M)}\otimes_BW \longrightarrow \]
	where $X_{N|M}$ is a complex appearing in a distinguished triangle
	\[M^\gp \longrightarrow N^\gp \longrightarrow X_{N|M} \longrightarrow  \quad .\]
	
	We have then a natural exact sequence
	\[\begin{tikzpicture}[baseline= (a).base]
	\node[scale=0.88] (a) at (0,0){
		\begin{tikzcd}[column sep=1em,row sep=0.5ex]
		\cdots \arrowr & H_n(\Z[M],\Z[N],W) \arrowr & H_n(A,B,W) \arrowr & H_n((A,M),(B,N),W) \arrowr & \cdots \\
		\cdots \arrowr & H_3(\Z[M],\Z[N],W) \arrowr & H_3(A,B,W) \arrowr & H_3((A,M),(B,N),W) \arrowr & \quad \\
		\quad \arrowr & H_2(\Z[M],\Z[N],W) \arrowr & H_2(A,B,W)\oplus\Tor_1^\Z(\ker(M^\gp\to N^\gp), W) \arrowr & H_2((A,M),(B,N),W) \arrowr & \quad \\
		\quad \arrowr & H_1(\Z[M],\Z[N],W) \arrowr & H_1(A,B,W)\oplus\Gamma \arrowr & H_1((A,M),(B,N),W) \arrowr & \quad \\
		\quad \arrowr & H_0(\Z[M],\Z[N],W) \arrowr & H_0(A,B,W)\oplus\big((N^\gp/\im(M^\gp\to N^\gp)\otimes_\Z W\big) \arrowr & H_0((A,M),(B,N),W) \arrowr & 0 ,
		\end{tikzcd}
	};
	\end{tikzpicture}\]
	where $\Gamma$ is a $B$-module appearing in a  (non canonically split) exact sequence
	\[
	0\to \ker(M^\gp\to N^\gp)\otimes_\Z W \to  \Gamma \to  \Tor_1^\Z(N^\gp/\im(M^\gp\to N^\gp), W)\to  0 ,
	\]
	and a natural exact sequence
	\[\begin{tikzpicture}[baseline= (a).base]
	\node[scale=0.87] (a) at (0,0){
		\begin{tikzcd}[column sep=1em,row sep=0.5ex]
		0 \arrowr & H^0((A,M),(B,N),W) \arrowr & H^0(A,B,W)\oplus \Hom_\Z(N^\gp/\im(M^\gp\to N^\gp),W)  \arrowr & H^0(\Z[M],\Z[N],W)  \arrowr & \quad\\
		\quad \arrowr & H^1((A,M),(B,N),W) \arrowr & H^1(A,B,W)\oplus\Lambda  \arrowr & H^1(\Z[M],\Z[N],W)  \arrowr & \quad \\
		\quad \arrowr & H^2((A,M),(B,N),W) \arrowr & H^2(A,B,W)\oplus\Ext_\Z^1(\ker(M^\gp \to N^\gp), W)  \arrowr & H^2(\Z[M],\Z[N],W)  \arrowr & \quad \\
		\quad \arrowr & H^3((A,M),(B,N),W) \arrowr & H^3(A,B,W)  \arrowr & H^3(\Z[M],\Z[N],W)  \arrowr & \quad\\
		\cdots  \arrowr & H^n((A,M),(B,N),W) \arrowr & H^n(A,B,W)  \arrowr & H^n(\Z[M],\Z[N],W) \arrowr & \cdots
		\end{tikzcd}
	};
	\end{tikzpicture}\]
	with a (non canonically split) exact sequence
	\[ 0 \longrightarrow  \Ext_\Z^1(N^\gp/\im(M^\gp\to N^\gp), W) \longrightarrow  \Lambda \longrightarrow   \Hom_\Z(\ker(M^\gp\to N^\gp),W) \longrightarrow  0 . \]
\end{theorem}
\begin{proof}
	By definition of the logarithmic module of differentials, we have a pushout for each $n$
	\[
	\begin{tikzcd}[row sep=4em, column sep=3em]
	\Omega_{\Z[R_n]|\Z[M]}\otimes_{\Z[R_n]} B \arrowr\arrowd{\alpha_n} & B\otimes_{\Z} R_n^\gp/M^\gp \arrowd{\gamma_n}\\
	\Omega_{F_n|A}\otimes_{F_n} B \arrowr & \Omega_{(F_n,R_n)|(A,M)}\otimes_{F_n} B
	\end{tikzcd}
	\]
	where $(A,M)\to(F_*,R_*)\to (B,N)$ is a factorization with $(A,M)\to(F_*,R_*)$ a free cofibration and $(F_*,R_*)\to(B,N)$ a trivial fibration. Since $A\to F_*\to B$ is a factorization cofibration - trivial fibration in the category of simplicial rings, $\Omega_{F_*|A}\otimes_{F_*}B=\LLL_{B|A}$ is the Andr\'e-Quillen cotangent complex of the $A$-algebra $B$. But $M\to R_n \to N$ is also a cofibration - trivial fibration in the category of simplicial monoids, and so by \cite[I.2.2.3, p.26]{Il-CC} and \cite[2.5, 2.6]{Quillen-MIT} $\Z[M]\to \Z[R_*]\to \Z[N]$ is a cofibration - trivial fibration in the category of simplicial rings.
	
	The maps $\alpha_n$ are injective and split since if $R_n=M\oplus \N^{Y_n}$ and $F_n=A[X_n,Y_n]$, then they are induced by the canonical injections
	\[\Omega_{\Z[R_n]|\Z[M]}\otimes_{\Z[R_n]}B= B^{Y_n}\longhookrightarrow B^{X_n\cup Y_n}=\Omega_{F_n|A}\otimes_{F_n}B .\] 

So we have a distinguished triangle
\[\Omega_{\Z[R]|\Z[M]}\otimes_{\Z[R]}W \longrightarrow  (\Omega_{F|A}\otimes_{F}W)\oplus(R^\gp/M^\gp\otimes_Z W) \longrightarrow \Omega_{(F,R)|(A,M)}\otimes_{F}W \longrightarrow \]
that is
\[\LLL_{\Z[N]|\Z[M]}\otimes_{\Z[N]}W \longrightarrow  (\LLL_{B|A}\otimes_{B}W)\oplus(R^\gp/M^\gp\otimes_Z W) \longrightarrow \LLL_{(B,N)|(A,M)}\otimes_{B}W \longrightarrow \ .\]

	It remains to compute $H_*(W\otimes_\Z R^\gp/M^\gp)$ and $H^*(\Hom_\Z(R^\gp/M^\gp,W))$. We have an exact\linebreak sequence, split in each degree
	\[0 \longrightarrow M^\gp \longrightarrow R^\gp \longrightarrow R^\gp/M^\gp \longrightarrow 0 .\]
	
	We also have $H_n(M^\gp)=0$ for $n>0$ and $H_0(M^\gp)=M^\gp$, since $M^\gp$ is concentrated in degree 0, and $H_n(R^\gp)=H_n(N^\gp)=0$ for $n>0$ and $H_0(R^\gp)=H_0(N^\gp)=N^\gp$ by \cite[Appendix A: Theorem A.5 and Remark A.6]{Ol}. So we have $H_n(R^\gp/M^\gp)=0$ for $n\geq2$ and an exact sequence
	\[
	0\longrightarrow H_1(R^\gp/M^\gp)\longrightarrow M^\gp \longrightarrow N^\gp \longrightarrow H_0(R^\gp/M^\gp)\longrightarrow 0 ,
	\]
	and then $H_1(R^\gp/M^\gp)=\ker(M^\gp \to N^\gp)$ and $H_0(R^\gp/M^\gp)=N^\gp/\im(M^\gp \to N^\gp)$.
	
	The result now follows from the universal coefficient exact sequences
	\[
	0\longrightarrow   W\otimes_\Z H_n( R^\gp/M^\gp) \longrightarrow   H_n(W\otimes_\Z R^\gp/M^\gp) \longrightarrow   \Tor_1^\Z(W, H_{n-1}(R^\gp/M^\gp)) \longrightarrow   0
	\]
	and
	\[0\longrightarrow  \Ext_\Z^1(H_{n-1}(R^\gp/M^\gp),W)\longrightarrow  H^n(\Hom_\Z(R^\gp/M^\gp,W)) \longrightarrow  \Hom_\Z(H_n( R^\gp/M^\gp), W)\longrightarrow  0 .\]
	
	Finally, the exact sequence $0 \to M^\gp \to R^\gp \to R^\gp/M^\gp \to 0 $ and the fact that $R^\gp$ is quasi-isomorphic to $N^\gp$ gives the desired triangle $M^\gp \to N^\gp \to R^\gp/M^\gp \to \ $.
\end{proof}

\begin{proposition}\label{4.04}
	Let $(A,M)\to (B,N)$ be a homomorphism essentially of finite type of noetherian prelog rings (see Definition~\ref{2.01}), let $C$ be a noetherian $B$-algebra and let $W$ be  a $C$-module of finite type. Then $H_n((A,M),(B,N),W)$ and $H^n((A,M),(B,N),W)$ are $C$-modules of finite type.
\end{proposition}
\begin{proof}
	It follows from Theorem~\ref{4.03} and the analogue result in Andr\'e-Quillen (co)homology \cite[4.55]{An-1974}.
\end{proof}

\begin{proposition}\label{4.05}
	Let $n\geq0$ be an integer. In the situation of Proposition~\ref{4.04}, the following are equivalent:
	\begin{enumerate}
		\item[(i)] $H_n((A,M),(B,N),W)=0$ for all $C$-modules $W$,
		\item[(ii)] $H^n((A,M),(B,N),W)=0$ for all $C$-modules $W$,
		\item[(iii)] $H_n((A,M),(B,N),C/\nnn)=0$ for all maximal ideals $\nnn$ of $C$,
		\item[(iv)] $H^n((A,M),(B,N),C/\nnn)=0$ for all maximal ideals $\nnn$ of $C$.
	\end{enumerate}
\end{proposition}
\begin{proof}
	Similar to that of \cite[4.57]{An-1974}.
\end{proof}

\begin{proposition}\label{4.06}
	Let $(A,M)\to (B,N)$ be a homomorphism of prelog rings and let $W$ be a $B$-module. Then
	\begin{enumerate}
		\item[(i)] $H_0((A,M),(B,N),W)=\Omega_{(B,N)|(A,M)}\otimes_BW$.
		\item[(ii)] $H^0((A,M),(B,N),W)=\Der_{(A,M)}((B,N),B)$.
	\end{enumerate}
	Moreover, if $A \to B$ and $M\to N$ are surjective, then
	\begin{enumerate}
		\item[(iii)] $H_1((A,M),(B,N),W)=N_{(B,N)|(A,M)}\otimes_BW$.
		\item[(iv)] $H^1((A,M),(B,N),W)=\Hom_B(N_{(B,N)|(A,M)},W)$.
	\end{enumerate}
\end{proposition}
\begin{proof}
	(i) $H_0(\Z[M],\Z[N],W)=\Omega_{Z[N]|\Z[M]}\otimes_{\Z[N]}W$ and $H_0(A,B,W)=\Omega_{B|A}\otimes_B N$ by \cite[6.3]{An-1974} and therefore the result follows from Theorem~\ref{4.03}  and the definition of the logarithmic module of differentials.
	
	(iii) $H_0(\Z[M],\Z[N],W)=0$, $H_1(\Z[M],\Z[N],W)=W\otimes_{\Z[N]} J/ J^2$ where $ J:=\ker(\Z[M]\to\Z[N])$ and $H_1(A,B,M)=W\otimes I/ I^2$ where $ I:=\ker(A\to B)$ \cite[6.1, 6.3]{An-1974}. With the notation of\linebreak Theorem~\ref{4.03}, $\Gamma=W\otimes_\Z\ker(M^\gp\to N^\gp)$. So again we obtain the isomorphism by definition of the conormal module.
	
	(ii) and (iv) are similar.
\end{proof}

\begin{proposition}\label{4.07}
	Let $(A,M)$ be a prelog ring and $(B,N)$ a free $(A,M)$-prelog algebra (Example~\ref{3.14}.(i)). Then
	\[H_n((A,M),(B,N),W)=0=H^n((A,M),(B,N),W)\]
	for all $n>0$ and all $B$-modules $W$.
\end{proposition}
\begin{proof}
	$(A,M)\to (B,N) \to (B,N)$ is a cofibration - trivial fibration.
\end{proof}

\begin{proposition}[Base Change]\label{4.08}
	Let $(A,M)\to(B,N)$ and $(A,M)\to(C,P)$ be homomorphisms of prelog rings, let $(D,Q):=(B\otimes_AC,N\oplus_MP)$, and let $W$ be a $B\otimes_AC$-module. Assume that the following conditions hold:
	\begin{enumerate}
		\item[(i)] $\Tor_i^A(B,C)=0$ for all $i>0$.
		\item[(ii)] $\Tor_i^{\Z[M]}(\Z[N],\Z[P])=0$ for all $i>0$.
		\item[(iii)] The canonical homomorphism $M^\gp\to N^\gp\oplus P^\gp$ is injective (e.g. if $M^\gp\to N^\gp$ or $M^\gp\to P^\gp$ is injective).
	\end{enumerate}
	Then the canonical homomorphisms
	\begin{align*}
	&H_n((A,M),(B,N),W)\longrightarrow H_n((C,P),(D,Q),W) ,\\
	&H^n((C,P),(D,Q),W)\longrightarrow H^n((A,M),(B,N),W)
	\end{align*}
	are isomorphisms for all $n\geq0$.
\end{proposition}
\begin{proof}
	The canonical homomorphisms
	\begin{align*}
	H_n(A,B,W)&\longrightarrow H_n(C,D,W) , \\
	H^n(C,D,W)&\longrightarrow H^n(A,B,W) , \\
	H_n(\Z[M],\Z[N],W) &\longrightarrow H_n(\Z[P],\Z[Q],W) , \\
	H^n(\Z[P],\Z[Q],W) &\longrightarrow H^n(\Z[M],\Z[N],W)
	\end{align*}
	are isomorphisms for all $n\geq0$ by \cite[4.54]{An-1974}.
	
	Since $(\ \ )^\gp$ is left adjoint, we have a pushout
	\[
	\begin{tikzcd}[column sep=4em,row sep=4em]
	M^\gp \arrowd{j} \arrowr{u} & N^\gp \arrowd{v} \\
	P^\gp \arrowr{w} & Q^\gp
	\end{tikzcd}
	\]
	Since $u\oplus j\colon M^\gp\to N^\gp\oplus P^\gp$ is injective, we have isomorphisms of abelian groups
	\begin{align*}
	\ker(u)&= \ker(w) ,\\
	\coker(u)&= \coker(w) .
	\end{align*}
	The result then follows from the fundamental exact sequences (Theorem~\ref{4.03}).
\end{proof}

\begin{proposition}[Jacobi-Zariski exact sequence]\label{4.09}
	Let $(A,M)\to(B,N)\to (C,L)$ be homomorphisms of prelog rings and $W$ a $C$-module. There exist natural exact sequences
	\[
	\begin{tikzcd}[column sep=1em,row sep=0ex]
	\cdots \arrowr & H_n((A,M),(B,N),W) \arrowr & H_n((A,M),(C,L),W) \arrowr & H_n((B,N),(C,L),W) \arrowr & \; \\
	\; \arrowr & H_{n-1}((A,M),(B,N),W) \arrowr & \cdots \arrowr & H_0((B,N),(C,L),W) \arrowr & 0
	\end{tikzcd}
	\]
	and
	\[
	\begin{tikzcd}[column sep=1em,row sep=0ex]
	0\arrowr & H^0((B,N),(C,L),W)\arrowr & H^0((A,M),(C,L),W)\arrowr & H^0((A,M),(B,N),W) \arrowr & \; \\
	\; \arrowr & H^1((B,N),(C,L),W) \arrowr & \cdots .
	\end{tikzcd}
	\]
\end{proposition}
\begin{proof}
	\cite[Theorem 8.18]{Ol}
\end{proof}

\begin{proposition}[Localization]\label{4.10}
	Let $(A,M)\to (B,N)$ be a homomorphism of prelog rings $W$ a $B$-module. Let $S\subset N$, $T\subset(B,\cdot)$ be submonoids such that $\alpha(S)\subset T$ where $\alpha\colon N \to B$ is the structural homomorphism of $(B,N)$. Then
	\begin{enumerate}
		\item[(i)] 
		$H_n((B,N),(T^{-1}B,S^{-1}N),T^{-1}W)=0=H^n((B,N),(T^{-1}B,S^{-1}N),T^{-1}W)$ for all $n\geq 0$.
		\item[(ii)] 
		$H_n((A,M),(B,N),T^{-1}W)=H_n((A,M),(T^{-1}B,S^{-1}N),T^{-1}W)$ and \\
		$H^n((A,M),(B,N),T^{-1}W)=H^n((A,M),(T^{-1}B,S^{-1}N),T^{-1}W)$ for all $n\geq 0$.
		\item[(iii)] 
		$H_n((A,M),(B,N),T^{-1}W)=T^{-1}H_n((A,M),(B,N),W)$ for all $n\geq 0$.
	\end{enumerate}
\end{proposition}
\begin{proof}
	\begin{enumerate}
		\item[(i)] 
		It follows from Base Change (Proposition~\ref{4.08}) applied to the pushout
		\[
		\begin{tikzcd}[row sep=4em, column sep=3em]
		(B,N) \arrowr\arrowd & (T^{-1}B,S^{-1}N)\arrowd\\
		(T^{-1}B,S^{-1}N) \arrowr & (T^{-1}B,S^{-1}N)
		\end{tikzcd}
		\]
		\item[(ii)] 
		It follows from (i) and the Jacobi-Zariski exact sequence (Proposition~\ref{4.09}).
		\item[(iii)] 
		Property (iii) follows from the universal coefficient spectral sequence
		\begin{align*}
		E_{p,q}^2&=\Tor_p^B(H_q(\LLL_{(B,N)|(A,M)}\otimes_BW),T^{-1}B)\\
		&=\Tor_p^B(H_q((A,M),(B,N),W),T^{-1}B) \;\Rightarrow\; H_n((A,M),(B,N),T^{-1}W) .
		\end{align*}
	\end{enumerate}
\end{proof}

\begin{remark}\label{4.11}
	Let $H\overset{i}{\to}X\overset{p}{\to}N$ be homomorphisms of simplicial monoids with $H,N$ constants, $H$ a group and $p i$ injective. We suppose that $p$ is a trivial fibration in the category of the simplicial monoids. Since $N$ is a Kan complex (because it is constant) and $p$ is a fibration, $X$ is a Kan complex too (the composition of fibrations is a fibration). We consider the induced actions of $H$ in $X$ and in $N$ and also the quotients $X/H$, $N/H$. We will assume that $H$ acts freely on $X$. Then $X\overset{\pi}{\to}X/H$ is a fibration \cite[18.2]{May} with fibre $H$, and then for any basepoint we have an exact sequence 
	\[ \cdots \longrightarrow \pi_n(H) \longrightarrow \pi_n(X) \longrightarrow \pi_n(X/H) \longrightarrow \pi_{n-1}(H)\longrightarrow \cdots \longrightarrow \pi_0(X/H) \longrightarrow 1 . \]
	
	We have $\pi_n(H)=\pi_n(X)=1$ for all $n>0$, and so we obtain $\pi_n(X/H)=1$ for all $n>1$ and an exact sequence
	\[ 1 \longrightarrow \pi_1(X/H) \longrightarrow \pi_0(H) \longrightarrow \pi_0(X) \longrightarrow \pi_0(X/H) \longrightarrow 1 \]
	which takes the form
	\[ 1 \longrightarrow \pi_1(X/H) \longrightarrow H \overset{p i}{\longrightarrow} N \longrightarrow \pi_0(X/H) \longrightarrow 1 . \]
	Therefore we deduce
	\[   
	\pi_n(X/H) = 
	\begin{cases}
	N/H &\quad\text{if } n=0 , \\
	1 &\quad\text{if } n>0 .
	\end{cases}
	\]
	
	That is, $X/H\to N/H$ is a weak equivalence, and then by \cite[I.2.2.3]{Il-CC} $\Z[X/H]\to \Z[N/H]$ is a weak equivalence.

\end{remark}

\begin{lemma}\label{4.12} Let $H$ be an abelian group and $W$ a $\Z[H]$-module. Then
	\begin{align*}
	H_n(\Z,\Z[H],W)&=\Tor_n^\Z(H,W) , \\
	H^n(\Z,\Z[H],W)&=\Ext_n^\Z(H,W)
	\end{align*}
	for all $n\geq 0$. In particular, these modules vanish for all $n\geq 2$.
\end{lemma}
\begin{proof}
	Let $X\to H$ be a simplicial resolution of the $\Z$-module $H$, that is, $X\to H$ is surjective, $X$ is a free simplicial $\Z$-module with $\pi_n(X)\cong\pi_n(H)$ for all $n$. Then $\Z[X]$ is a free simplicial $\Z$-algebra, $\Z[X]\to \Z[H]$ is surjective and a weak equivalence by \cite[I.2.2.3]{Il-CC}. So $\Z\to \Z[X]\to\Z[H]$ is a free cofibration - trivial fibration factorization of simplicial rings \cite[Proposition 2.5]{Quillen-MIT}. Therefore $H_n(\Z,\Z[H],W)=H_n(\Omega_{\Z[X]|\Z}\otimes_{\Z[X]}W)$ and similarly for cohomology.
	
	So it suffices to show that $\Omega_{\Z[X]|\Z}\simeq X\otimes_\Z\Z[X]$. Each $X_n$ is a free $\Z$-module. Since $\Omega$, $\otimes$ and $\Z[-]$ commute with colimits, we can assume that $X_n$ is a free $\Z$-module of finite rank and by K\"{u}nneth formula for the module of differentials \cite[1.14]{An-1974} we can assume $X_n\simeq \Z$:
	\begin{align*}
	\Omega_{\Z[H_1\oplus H_2]|\Z} &= \Omega_{\Z[H_1]\otimes_\Z\Z[H_2]|\Z}\\
	&= \Omega_{\Z[H_1]|\Z}\otimes_\Z\Z[H_2]\oplus\Z[H_1]\otimes_\Z\Omega_{\Z[H_2]|\Z}\\
	&=H_1\otimes_\Z\Z[H_1]\otimes_\Z\Z[H_2]\oplus \Z[H_1]\otimes_\Z H_2 \otimes_\Z\Z[H_2]\\
	&=(H_1\oplus H_2) \otimes_\Z(\Z[H_1]\otimes_\Z\Z[H_2])\\
	&=(H_1\oplus H_2) \otimes_\Z \Z[H_1\oplus H_2] .
	\end{align*}
	
	So we only have to prove $\Omega_{\Z[H]|\Z}=H\otimes_\Z\Z[H]$ when $H=\Z$. We have $\Z[H]=\Z[t,\frac{1}{t}]=S^{-1}\Z[t]$ with $S=\{1,t,t^2,\ldots\}$. Then $\Omega_{\Z[H]|\Z}=S^{-1}\Omega_{\Z[t]|\Z}=S^{-1}\Z[t]=\Z\otimes_\Z S^{-1}\Z[t]=H\otimes_\Z\Z[H]$.
\end{proof}

\begin{lemma}\label{4.13}
	Let $(A,H)\to(B,G)$ be a homomorphism of prelog rings, where $H$ and $G$ are groups. Then $H_n((A,H),(B,G),W)=H_n(A,B,W)$ for all $n\geq 0$ and all $B$-modules $W$, and similarly for cohomology.
\end{lemma}
\begin{proof}
	We have a Jacobi-Zariski exact sequence
	\[
	\begin{tikzcd}[column sep=1em,row sep=0ex]
	& & & H_{n+1}((B,H),(B,G),W) \arrowr & \;\\
	\;\arrowr & H_n((A,H),(B,H),W) \arrowr & H_n((A,H),(B,G),W) \arrowr & H_n((B,H),(B,G),W) \arrowr & \cdots
	\end{tikzcd}
	\]
	Since $H_n((A,H),(B,H),W)=H_n(A,B,W)$ for all $n\geq 0$ by Example~\ref{4.02}, it suffices to show that $H_n((B,H),(B,G),W)=0$ for all $n\geq 0$. From the Jacobi-Zariski sequence associated to
	$(B,\{1\})\to(B,H)\to (B,G)$, we see that it suffices to show that $H_n((B,\{1\}),(B,H),W)=0$ for any group $H$ and all $n\geq 0$.
	
	Since $H_n(B,B,W)=0$ for all $n\geq 0$, the fundamental exact sequence gives isomorphisms
	\[ H_n((B,\{1\}),(B,H),W)=H_{n-1}(\Z,\Z[H],W)=0\] for $n\geq 3$ by Lemma~\ref{4.12}.(ii)  and an exact sequence
	\[\begin{array}{llllllll}
	&&&\quad\quad0&\longrightarrow&H_2((B,\{1\}),(B,H),W)&\longrightarrow&\\
	\longrightarrow&H_1(\Z,\Z[H],W)&\overset{\kappa_1}{\longrightarrow}&\Tor_1^\Z(W,H)&\longrightarrow&H_1((B,\{1\}),(B,H),W)&\longrightarrow&\\
	\longrightarrow&H_0(\Z,\Z[H],W)&\overset{\kappa_0}{\longrightarrow}&\ \  W\otimes_\Z H&\longrightarrow&H_0((B,\{1\}),(B,H),W)&\longrightarrow&0 . \\
	\end{array} \]
	Furthermore, since $\kappa_1$ and $\kappa_0$ are isomorphisms by Lemma~\ref{4.12} we deduce also $H_i((B,\{1\}),(B,H),W)=0$ for $i=0,1,2$.
\end{proof}

\begin{proposition}\label{4.14}\cite[Theorem 8.20]{Ol}
	Let $(A,M)\to(B,N)$ be a homomorphism of prelog rings with $M$ and $N$ integral monoids and $W$ a $B$-module. Then, for all $n\geq 2$, we have
	\begin{enumerate}
		\item[(i)] $H_n((A,M),(B,N),W)=H_n((A,M),(B,N)^\llog,W)=H_n((A,M)^\llog,(B,N)^\llog,W)$.
		\item[(ii)] $H^n((A,M),(B,N),W)=H^n((A,M),(B,N)^\llog,W)= H^n((A,M)^\llog,(B,N)^\llog,W)$.
	\end{enumerate}
\end{proposition}
\begin{proof}
	We will prove (i). By the Jacobi-Zariski exact sequences, it suffices to show that for a prelog ring $(A,M)$ with $M$ integral, $H_n((A,M),(A,M)^\llog,W)=0$ for any $A$-module $W$ for all $n\geq 0$. Let $\alpha\colon M\to A$ be the structural map. Consider the diagram of pushouts
	\[
	\begin{tikzcd}[column sep=3em,row sep=4em]
	(A,\alpha^{-1}(A^*)) \arrowr\arrowd & (A,\alpha^{-1}(A^*)^\gp) \arrowr\arrowd & (A,\alpha^{-1}(A^*)^\gp/\ker v) \arrowr\arrowd & (A,A^*) \arrowd \\
	(A,M) \arrowr & (A,M_1) \arrowr & (A,M_2) \arrowr & (A,M^\llog)
	\end{tikzcd}
	\]
	where $v\colon \alpha^{-1}(A^*)^\gp \to A^*$ is the canonical map. Applying base change (Proposition~\ref{4.08}) to the pushout
	\[
	\begin{tikzcd}[column sep=3em,row sep=4em]
	(A,\alpha^{-1}(A^*)) \arrowr\arrowd & (A,\alpha^{-1}(A^*)^\gp) \arrowd \\
	(A,M) \arrowr & (A,M_1)
	\end{tikzcd}
	\]
	(since $\Tor_n^{\Z[\alpha^{-1}(A^*)]}(\Z[M],\Z[\alpha^{-1}(A^*)^\gp]))=0$ for $n>0$ because we know that $\Z[\alpha^{-1}(A^*)]\longrightarrow \Z[\alpha^{-1}(A^*)^\gp]$ being a localization is flat), we obtain
	\[ H_n((A,M),(A,M_1),W)= H_n((A,\alpha^{-1}(A^*)),(A,\alpha^{-1}(A^*)^\gp),W)=0 \]
	where this last module vanishes by Proposition~\ref{4.10}.(ii).
	
	Now consider to the pushout
	\[
	\begin{tikzcd}[column sep=3em,row sep=4em]
	(A,\alpha^{-1}(A^*)^\gp) \arrowr\arrowd & (A,\alpha^{-1}(A^*)^\gp/\ker v) \arrowd \\
	(A,M_1) \arrowr & (A,M_2)
	\end{tikzcd}
	\]
	since $M$ is integral, the map $\alpha^{-1}(A^*)^\gp \to M_1$ is injective, being a localization of the injective map $\alpha^{-1}(A^*)\to M$. In particular, $\ker(v)\to M_1$ is injective. If we take a cofibration-trivial fibration factorization
	\[ \alpha^{-1}(A^*)^\gp \longrightarrow X \longrightarrow M_1 \]
	in the category of simplicial monoids with $X_n=\alpha^{-1}(A_1^*)^\gp\oplus\N^{\widetilde{X}_n}$ a free $\alpha^{-1}(A^*)^\gp$-monoid, then the action of $\ker(v)$ on $X$ is free, since $\alpha^{-1}(A^*)^\gp$ is a group. Therefore from \cite[I.2.2.3]{Il-CC} and Remark~\ref{4.11}:
	\begin{align*}
	\Tor_n^{\Z[\alpha^{-1}(A^*)^\gp]}(\Z[M_1],\Z[\alpha^{-1}(A^*)^\gp/\ker(v)]) &= H_n(\Z[X]\otimes_{\Z[\alpha^{-1}(A^*)^\gp]}\Z[\alpha^{-1}(A^*)^\gp/\ker(v)]) \\
	&= H_n(\Z[X/\ker(v)])=0
	\end{align*}
	for $n>0$.
	
	So base change gives an isomorphism
	\[ H_n((A,M_1),(A,M_2),W)= H_n((A,\alpha^{-1}(A^*)^\gp),(A,\alpha^{-1}(A^*)^\gp/\ker(v)),W)=0 \]
	where the last module vanishes by Lemma~\ref{4.13} above.
	
	Finally, applying base change to the pushout
	\[
	\begin{tikzcd}[column sep=3em,row sep=4em]
	(A,\alpha^{-1}(A^*)^\gp/\ker(v)) \arrowr\arrowd & (A,A^*) \arrowd \\
	(A,M_2) \arrowr & (A,M^\llog)
	\end{tikzcd}
	\]
	(since $\alpha^{-1}(A^*)^\gp/\ker(v)\to A^*$ is an injective homomorphism of groups, we know that the homomorphism $\Z[\alpha^{-1}(A^*)^\gp/\ker(v)] \to \Z[A^*]$ is flat) we obtain similarly
	\[ H_n((A,M_2),(A,M^\llog),W)= H_n((A,\alpha^{-1}(A^*)^\gp/\ker(v)),(A,A^*),W)=0 \]
	which vanishes again by Lemma~\ref{4.13}.
	
	By the Jacobi-Zariski exact sequence associated to $(A,M)\longrightarrow (A,M_1)\longrightarrow (A,M_2)$ we obtain $H_n((A,M),(A,M_2),W)=0$, and then by the Jacobi-Zariski exact sequence associated to $(A,M)\to (A,M_2)\to (A,M^\llog)$ we obtain $H_n((A,M),(A,M^\llog),W)=0$ as desired.
\end{proof}

\section{Smoothness}\label{S}
\begin{definition}\label{5.01}
	\cite{Il-GL, Og}
	Let $(A,M)\to(B,N)$ be a homomorphism of prelog rings where $(B,N)$ is a log ring. An \emph{$(A,M)$-extension of $(B,N)$ by $I$} is a strict homomorphism of log rings $\tau\colon (C,P)\to(B,N)$ such that $\tau^\sharp\colon C\to B$ is surjective with kernel $I$, verifying $I^2=0$ ($\tau^\flat\colon P\to N$ is then surjective) and the subgroup $1+I$ of $C^*$ acts freely on $P$.
\end{definition}

\begin{remark}\label{5.02}
	Let $\tau:(C,P)\to(B,N)$ be a homomorphism of log rings with $\tau^\sharp\colon C\to B$ surjective and kernel $I$. If $P$ is integral, then $C^*$ (and so $1+I$) acts freely on $P$.
\end{remark}

\begin{definition}\label{5.03}
	Let $(A,M)\to (B,N)$ be a homomorphism of prelog rings and $J$ an ideal of $B$. We say that $(B,N)$ is \emph{log formally smooth} over $(A,M)$ for the $J$-adic topology if for any commutative diagram of prelog rings
	\[
	\begin{tikzcd}[column sep=4em,row sep=4em]
	(A,M) \arrowr\arrowd & (B,N) \arrowd{f} \\
	(C,P) \arrowr & (C',P')
	\end{tikzcd}
	\]
	where $(C,P)\to (C',P')$ is an $(A,M)$-extension and $f(J^t)=0$ for some $t>0$, there exist a homomorphism of prelog rings $h\colon (B,N)\to (C,P)$ making commutative the triangles.
\end{definition}

\begin{proposition}\label{5.04}
	Let $(A,M)\to(B,N)$ be a homomorphism of prelog rings and $J$ an ideal of $B$. The following are equivalent:
	\begin{enumerate}
		\item[(i)] $(B,N)$ is a log formally smooth $(A,M)$-algebra for the $J$-adic topology,
		\item[(ii)] $(B,N)^\llog$ is a log formally smooth $(A,M)$-algebra for the $J$-adic topology,
		\item[(iii)] $(B,N)^\llog$ is a log formally smooth $(A,M)^\llog$-algebra for the $J$-adic topology.
	\end{enumerate}
\end{proposition}
\begin{proof}
	It follows easily from the universal property of the functor $(\ \ )^\llog$ and the diagram
	\[
	\begin{tikzcd}[column sep=4em,row sep=4em]
	(A,M) \arrowr\arrowd & (B,N) \arrowd \\
	(A,M)^\llog \arrowr\arrowd & (B,N)^\llog \arrowd \\
	(C,P) \arrowr & (C',P')
	\end{tikzcd}
	\]
	where $(C,P)\to(C',P')$ is an $(A,M)$-extension and $f(J^t)=0$ for some $t$. 
\end{proof}

\begin{proposition}\label{5.05}
	Let $\tau:(C,P)\to(B,N)$ be an $(A,M)$-extension of $(B,N)$ by $I$. Then the map
	\begin{align*}
	(1+I)\times P \;&\longrightarrow\; P\times_NP \\
	(u,p) \;&\longmapsto\; (p,up)
	\end{align*}
	is an isomorphism.
\end{proposition}
\begin{proof}
	\cite[Expos\'{e} II, Proposition 2.3.(b)]{Il-GL} or \cite[Proposition IV.2.1.2.3]{Og}.
\end{proof}

\begin{definition}\label{5.06}
	\cite{Il-GL}
	Let $\tau:(C,P)\to(B,N)$ be an $(A,M)$-extension of $(B,N)$ by $I$ and $\tau'\colon (C',P')\to(B,N)$ an $(A,M)$-extension of $(B,N)$ by $I'$. A \emph{morphism of extensions} is a homomorphism of log rings $h:(C,P)\to(C',P')$ verifying $\tau'h=\tau$. We have then $h^\flat(up)=uh^\flat(p)$ for all $u\in 1+I$ and $p\in P$ (the action of $1+I$ on $P'$ is via $1+I\to1+I'$). 
	
	An \emph{isomorphism of extensions} is a morphism with inverse.
\end{definition}

\begin{definition}\label{5.07}
	Let $\tau:(C,P)\to(B,N)$ be an $(A,M)$-extension of $(B,N)$ by $W$ and let $(B',N')\to(B,N)$ a homomorphism of log rings over $(A,M)$. Then we have a commutative diagram whose rows are $(A,M)$-extensions by $W$
	\[
	\begin{tikzcd}[column sep=3em,row sep=4em]
	(C\times_BB',P\times_NN') \arrowr\arrowd & (B',N') \arrowd \\
	(C,P) \arrowr & (B,N)
	\end{tikzcd}
	\]
\end{definition}

\begin{proposition}\label{5.08}
	Let $\tau:(C,P)\to(B,N)$, $\tau':(C',P')\to(B,N)$ be $(A,M)$-extensions of $(B,M)$ by $I$, and let $h:(C,P)\to(C',P')$ a morphism of extensions. Then $h$ is an isomorphism of extensions.
\end{proposition}
\begin{proof}
	We will see that $h^\flat$ is bijective. The result then follows easily. Let $p_1,p_2\in P$ be such that $h^\flat(p_1)=h^\flat(p_2)$. Then $\tau^\flat(p_1)=\tau^\flat(p_2)$ and so $(p_1,p_2)\in P\times_NP$. By Proposition~\ref{5.05}, there exists $u\in 1+I$ such that $p_2=up_1$. Since $h^\flat(p_1)=h^\flat(p_2)=h^\flat(up_1)=uh^\flat(p_1)$ and $1+I$ acts freely on $P$, we deduce $u=1$ and so $p_2=p_1$. This proves injectivity. Now let $p'\in P'$. Let $p\in P$ such that $\tau^\flat(p)=\tau'^\flat(p')$, and let $p''=h^\flat(p)$. We have $(p',p'')\in P'\times_NP'$ and then by Proposition~\ref{5.05} there exists $u\in 1+I$ such that $p'=up''$. Then $h^\flat(up)=uh^\flat(p)=up''=p'$ proving surjectivity.
\end{proof}

\begin{definition}\label{5.09}
	Let $(A,M)\to(B,N)$ be a homomorphism of prelog rings with $(B,N)$ a log ring and let $W$ be a $B$-module. We define the \emph{trivial $(A,M)$-extension of $(B,N)$ by $W$} as $\tau:(C,P)\to(B,N)$ where $C:=B\oplus W$ with product $(b_1,w_1)(b_2,w_2):=(b_1b_2,b_1w_2+b_2w_1)$, $P:=N\oplus(W,+)$ the product monoid, the structural map $\alpha_C:P\to C$ given by $\alpha_C(n,w)=(\alpha_B(n),\alpha_B(n)w)$, and $\tau^\sharp$, $\tau^\flat$ are the canonical projections. We have $C^*=\{(b,w) \mid b\in B^*\}$ since for $b\in B^*$, $(b,w)^{-1}=(b^{-1},-b^{-2}w)$. Since $(B,N)$ is a log ring, $\alpha_{B_{|\alpha_B^{-1}(B^*)}}\colon \alpha_B^{-1}(B^*) \to B^*$ is an isomorphism, and then $\alpha_C^{-1}(C^*)=\alpha_B^{-1}(B^*)\oplus W$. So $\alpha_{C_{|\alpha_C^{-1}(C^*)}}\colon \alpha_C^{-1}(C^*)\to C^*$ is bijective and $(C,P)$ is a log ring. Moreover, $\tau$ is strict and $1+W$ acts freely on $P$. 
\end{definition}

\begin{theorem}\label{5.10}
	Let $(A,M)\to(B,N)$ be a homomorphism of prelog rings with $(B,N)$ a log ring and $W$ a $B$-module. There exists a bijection, natural on $(B,N)$ in the sense of Definition~\ref{5.07}, between isomorphisms classes of $(A,M)$-extensions of $(B,N)$ by $W$ and $H^1((A,M),(B,N),W)$, where the trivial extension goes to zero.
\end{theorem}
\begin{proof}
	Let $(F,R)\overset{\pi}{\longrightarrow}(B,N)$ be a homomorphism of prelog rings with $\pi^\sharp$, $\pi^\flat$ surjective and  $(F,R)$ a log free $(A,M)$-algebra (see Example~\ref{3.14}.(i)). Let $I=\ker(\pi^\sharp: F\to B)$, $J=\ker(Z[R]\to\Z[N])$, $T=\ker(R^\gp\to N^\gp)$. We have an exact sequence of $B$-modules:
	\[
	\Der_{(A,M)}((F,R),W)\overset{\theta}{\to}
	\Hom_B(N_{(B,N)|(F,R)},W)\to
	H^1((A,M),(B,N),W)\to 0
	\]
	by propositions~\ref{4.09}, \ref{4.07} and~\ref{4.06}.
	
	Let $(C,P)\overset{\tau}{\longrightarrow}(B,N)$ be an $(A,M)$-extension of $(B,N)$ by $W$. Since $(F,R)$ is log free over $(A,M)$, we have a commutative diagram of prelog rings
	\[
	\begin{tikzcd}[column sep=4em,row sep=4em]
	(F,R) \arrowr{\pi}\arrowd{\varphi} & (B,N) \arrow[equal]{d} \\
	(C,P) \arrowr{\tau} & (B,N)
	\end{tikzcd}
	\]
	inducing an  $F$-module homomorphism ${\varphi^\sharp}_{|I}\colon I\to W$ and then a $B$-module homomorphism $\varphi_1\colon I/I^2\to W$. We define also a homomorphism of $B$-modules
	\[ \varphi_2\colon B\otimes_\Z T \longrightarrow W \]
	as follows. Let $\frac{r_1}{r_2}\in T$ (i.e. $r_1,r_2\in R$ with $\pi^\flat(r_1)=\pi^\flat(r_2)$). Since $\tau^\flat\big(\varphi^\flat(r_1)\big)=\tau^\flat\big(\varphi^\flat(r_2)\big)$, we have $\big(\varphi^\flat(r_1),\varphi^\flat(r_2)\big)\in P\times_NP$ and so by Proposition~\ref{5.05} there exists $u\in 1+W$ such that $\varphi^\flat(r_1)=u\varphi^\flat(r_2)$. This $u$ is unique since the action of $1+W$ on $P$ is free. We define $\varphi_2\big(1\otimes\frac{r_1}{r_2}\big):=u-1$.
	
	Consider the square
	\[
	\begin{tikzcd}[column sep=4em,row sep=4em]
	B\otimes_{\Z[N]}J/J^2 \arrowr{\nu}\arrowd{\eta} & B\otimes_\Z T \arrowd{\varphi_2} \\
	I/I^2 \arrowr{\varphi_1} & W
	\end{tikzcd}
	\]
	where $\nu$, $\eta$ are the maps defined in Definition~\ref{3.04}. This square is commutative since if $r_1,r_2\in R$ and $\pi^\flat(r_1)=\pi^\flat(r_2)$, then
	\begin{align*}
	-\varphi_2\nu\big(1\otimes(r_1-r_2)\big)&=\varphi_2\big(\alpha_B\pi^\flat(r_2)\otimes\frac{r_1}{r_2}\big)=\alpha_B\pi^\flat(r_2)\cdot(u-1)=\alpha_C\varphi^\flat(r_2)\cdot(u-1) \\
	&=u\alpha_C\varphi^\flat(r_2)-\alpha_C\varphi^\flat(r_2)=\alpha_C\varphi^\flat(r_1)-\alpha_C\varphi^\flat(r_2)
	\end{align*}
	where $u$ is chosen as above, and
	\[
	-\varphi_1\eta\big(1\otimes(r_1-r_2)\big)=\varphi^\sharp\big(\alpha_F(r_1)-\alpha_F(r_2)\big)=\alpha_C\varphi^\flat(r_1)-\alpha_C\varphi^\flat(r_2) .
	\]
	Since $B\otimes_{\Z[N]}J/J^2$ is generated as $B$-module by the elements $1\otimes(r_1-r_2)$ as before, we are done.
	
	Then by Definition~\ref{3.04}, $(\varphi_1,\varphi_2)$ determine an element $\widetilde{\varphi}\in\Hom_B(N_{(B,N)|(F,R)},W)$.
	
	Let now $\psi\colon (F,R)\to (B,N)$ be another homomorphism of prelog rings such that $\tau\psi=\pi$. We will see that $\widetilde{\psi}-\widetilde{\varphi}\in\im(\theta)$.
	
	We consider the log derivation $(D,\widetilde{D})\in\Der_{(A,M)}((F,R),W)$ given by the $A$-derivation
	\[ D=\psi^\sharp-\varphi^\sharp\colon F\longrightarrow W \]
	and the homomorphism of groups
	\[\widetilde{D} \colon R^\gp/\im(M^\gp\to R^\gp) \longrightarrow (W,+) \]
	defined as follows:
	Let $\frac{r_1}{r_2}\in R^\gp$. Since $\big(\psi^\flat(r_i),\varphi^\flat(r_i)\big)\in P\times_NP$ for $i=1,2$, there exists unique elements $u_1,u_2\in 1+W$ such that $u_i\varphi^\flat(r_i)=\psi^\flat(r_i)$. We define $\widetilde{D}\big(\overline{\frac{r_1}{r_2}}\big)=u_1-u_2\in W$. If $\frac{r_1}{r_2}\in \im(M^\gp\to R^\gp)$, then $u_1=1=u_2$ and so $\widetilde{D}$ is well defined. It is a group homomorphism since $W$ is of square zero as ideal of $C$. If $r\in R$, $\widetilde{D}(\bar{r})=u-1$ where $u\in1+W$ is such that $u\varphi^\flat(r)=\psi^\flat(r)$. Then
	\begin{align*}
	\alpha_B\pi^\flat(r)\widetilde{D}(\bar{r})&=\alpha_B\pi^\flat(r)\cdot(u-1)=\alpha_C\varphi^\flat(r)\cdot(u-1)=\alpha_C\varphi^\flat(r)u-\alpha_C\varphi^\flat(r) \\
	&=\alpha_C\psi^\flat(r)-\alpha_C\varphi^\flat(r)=\psi^\sharp\alpha_F(r)-\varphi^\sharp\alpha_F(r)=D\big(\alpha_F(r)\big)
	\end{align*}
	and so $(D,\widetilde{D})\in\Der_{(A,M)}((F,R),W)$. Finally, $\theta(D,\widetilde{D})=\widetilde{\psi}-\widetilde{\varphi}$.
	
	Conversely, let $h\in\Hom_B(N_{(B,N)|(F,R)},W)$, that is, $h$ is defined by a commutative diagram
	\[
	\begin{tikzcd}[column sep=3em,row sep=4em]
	B\otimes_{\Z[N]}J/J^2 \arrowr{\nu}\arrowd{\eta} & B\otimes_\Z T \arrowd{h_2} \\
	I/I^2 \arrowr{h_1} & W
	\end{tikzcd}
	\]
	Let $h_1'\colon I\to W$ the $F$-module homomorphism induced by $h_1$ and $C=F\oplus_IW$ the pushout
	\[
	\begin{tikzcd}[column sep=4em,row sep=4em]
	I \arrowr\arrowd{h_1'} & F \arrowd \\
	W \arrowr & C
	\end{tikzcd}
	\]
	with the usual ring structure: $(f_1,w_1)\cdot(f_2,w_2)=(f_1f_2,f_1w_1+f_2w_2)$. Let $h_2'\colon T\to W$
	be the homomorphism of abelian groups induced by $h_2$. The pushout
	\[
	\begin{tikzcd}[column sep=3em,row sep=4em]
	T \arrowr\arrowd{h_2'} & R^\gp \arrowd \\
	W \arrowr & R^\gp\oplus_TW
	\end{tikzcd}
	\]
	gives a commutative diagram of exact sequences of abelian groups
	\begin{center}
		\begin{tikzpicture}[scale=2]
		\node (B) at (1,1) {$1$};
		\node (C) at (2,1) {$T$};
		\node (D) at (3.25,1) {$R^\gp$};
		\node (E) at (4.5,1) {$N^\gp$};
		\node (F) at (5.5,1) {$1$};
		\node (H) at (1,0) {$0$};
		\node (I) at (2,0) {$W$};
		\node (J) at (3.25,0) {$R^\gp\oplus_TW$};
		\node (K) at (4.5,0) {$N^\gp$};
		\node (L) at (5.5,0) {$1$};
		\path[->,font=\scriptsize]
		(B) edge node[above]{} (C)
		(C) edge node[above]{} (D)
		(D) edge node[above]{} (E)
		(E) edge node[above]{} (F)
		(H) edge node[above]{} (I)
		(I) edge node[above]{} (J)
		(J) edge node[above]{$\varepsilon$} (K)
		(K) edge node[right]{} (L)
		(C) edge node[] {} (I)
		(D) edge node[right]{} (J);
		\draw[double equal sign distance] (E) to node {} (K);
		\end{tikzpicture}
	\end{center}
	Let $P$ be the pullback of monoids
	\[
	\begin{tikzcd}[column sep=3em,row sep=4em]
	P \arrowr\arrowd & R^\gp\oplus_TW \arrowd{\varepsilon} \\
	N \arrowr & N^\gp
	\end{tikzcd}
	\]
	We define $\alpha_C\colon P\to C=F\oplus_IW$ as follows: Let $(n,(\frac{r_1}{r_2},w))\in N\times_{N^\gp}(R^\gp\oplus_TW)=P$. Then $\varepsilon\big(\frac{r_1}{r_2},w\big)=\frac{\pi^\flat(r_1)}{\pi^\flat(r_2)}=\frac{n}{1}$. Let $r\in R$ be such that $\pi^\flat(r)=n$. Then (denoting the images of $r\in R$ and $n\in N$ in $B$ also by $r$ an $n$) we define
	\[ \alpha_C\big(n,\big(\frac{r_1}{r_2},w\big)\big):=\big(\alpha_F(r),h_2\big(r\otimes \frac{rr_2}{r_1}\big)-nw\big) . \]
	We will see first that this definition is independent of the choice of $r$ such that $\pi^\flat(r)=n$. Let $r,r'\in R$ be such that $\pi^\flat(r)=\pi^\flat(r')=n$. Then
	\begin{align*}
	&\big(\alpha_F(r'),h_2\big(r'\otimes \frac{r'r_2}{r_1}\big)-nw\big)-\big(\alpha_F(r),h_2\big(r\otimes \frac{rr_2}{r_1}\big)-nw\big) = \\
	&\qquad = \big(\alpha_F(r')-\alpha_F(r),nh_2\big(1\otimes \frac{r'r_2}{r_1}\cdot\frac{r_1}{rr_2}\big)\big) = \big(\alpha_F(r')-\alpha_F(r),nh_2\big(1\otimes \frac{r'}{r}\big)\big) \\
	&\qquad = \big(0,nh_2\big(1\otimes \frac{r'}{r}\big)-h_1'\big(\alpha_F(r')-\alpha_F(r)\big)\big)=\big(0,h_2\nu(r'-r)-h_1\eta(r'-r)\big)=0		
	\end{align*}
	since $h_2\nu=h_1\eta$.
	
	Now assume that $\big(n,(\frac{r_1}{r_2},w)\big)=\big(n,(\frac{r_1'}{r_2'},w')\big)\in N\times_{N^\gp}(R^\gp\oplus_TW)$ and we will see that
	\[
	\big(\alpha_F(r),h_2\big(r\otimes \frac{rr_2}{r_1}\big)-nw\big)=\big(\alpha_F(r),h_2\big(r\otimes \frac{rr_2'}{r_1'}\big)-nw'\big) .
	\]
	Let $\frac{s_1}{s_2}\in T$ be such that
	$\frac{r_1'}{r_2'}=\frac{r_1}{r_2}\frac{s_1}{s_2}$, $w'=h_2(1\otimes\frac{s_2}{s_1})+w$. Then
	\begin{align*}
	\big(\alpha_F(r),h_2\big(r\otimes\frac{rr_2'}{r_1'}\big)-nw'\big)&=	\big(\alpha_F(r),h_2\big(r\otimes\frac{rr_2s_2}{r_1s_1}\big)-nh_2\big(1\otimes\frac{s_2}{s_1}\big)-nw\big) \\
	&= \big(\alpha_F(r),h_2\big(r\otimes\frac{rr_2s_2}{r_1s_1}+r\otimes\frac{s_1}{s_2}\big)-nw\big) \\
	&= \big(\alpha_F(r),h_2\big(r\otimes\frac{rr_2}{r_1}\big)-nw\big) .
	\end{align*}
	
	Finally, it is immediate to check that $\alpha_C$ is a monoid homomorphism.
	
	Now we define
	\[ \tau\colon (C,P) \longrightarrow (B,N) \]
	by $\tau^\sharp(f,w):=\pi^\sharp(f)$ and $\tau^\flat(n,(\frac{r_1}{r_2},w)):=n$.
	
	We will see now that $\tau^\llog\colon (C,P^\llog)\to(B,N)$ is strict.
	
	Consider the pushout
	\[
	\begin{tikzcd}[column sep=2em,row sep=4em]
	(\tau\alpha_C)^{-1}(B^*) \arrowr\arrowd & B^* \arrowd \\
	P \arrowr & P\oplus_{(\tau\alpha_C)^{-1}(B^*)}B^*
	\end{tikzcd}
	\]
	We will see that the canonical map $\lambda\colon P\oplus_{(\tau\alpha_C)^{-1}(B^*)}B^* \to N $ has an inverse $\mu\colon N\to P\oplus_{(\tau\alpha_C)^{-1}(B^*)}B^*$. Note first that $((1,(1,w)),1)=((1,(1,0)),1)$ in $(N\times_{N^\gp}(R^\gp\oplus_TW))\oplus_{(\tau\alpha_C)^{-1}(B^*)}B^*$, since $\tau\alpha_C(1,(1,w))=1\in B^*$.
	
	Let $n\in N$. Let $r\in R$ such that $\pi^\flat(r)=n$. We define
	\[ \mu(n)=\big((n,(r,0)),1\big)\in (N\times_{N^\gp}(R^\gp\oplus_TW))\oplus_{(\tau\alpha_C)^{-1}(B^*)}B^* . \]
	If $\pi^\flat(r)=\pi^\flat(r')$, then $\frac{r}{r'}\in T$ and so $(r',0)=(r'\frac{r}{r'},h_2'(\frac{r'}{r}))=(r,h_2'(\frac{r'}{r}))$ in $R^\gp\oplus_TW$ and so
	\[ \big((n,(r',0),1\big)=\big(\big(n,\big(r,h_2'\big(\frac{r'}{r}\big)\big)\big),1\big)=\big((n,(r,0)),1\big) . \]
	
	Therefore $\mu$ does not depend on the choice of $r$. It is surjective since
	\begin{align*}
	\big(\big(n,\big(\frac{r_1}{r_2},w\big)\big),b\big) &= \big(\big(bn,\big(\frac{\tilde{b}r_1}{r_2},w\big)\big),1\big) = \big(\big(bn,\big(\frac{\tilde{b}r_1}{r_2},0\big)\big),1\big) \\
	&= \big((bn,(r',0)),1\big)
	\end{align*}
	where $r'=\frac{br_1}{r_2}\in R$ is such that $\pi^\flat(r')=bn$, we have identified $b\in B^*$ with its image via $B^*=\alpha_B^{-1}(B^*)\subset N$ and $\tilde{b}\in R$ is such that $\pi^\flat(\tilde{b})=b\in N$. It is injective since $\lambda\mu=\id_N$.
	
	Now let us see that $1+W$ acts freely on $P^\llog$. Since $\tau^\sharp\colon C\to B$ is surjective with square zero kernel, we have $C^*=(\tau^\sharp)^{-1}(B^*)=\{(f,x)\in F\oplus_{I}W \mid \pi^\sharp(f)\in B^* \}$. Let $[(n,(\frac{r_1}{r_2},w)),(f,x)]\in \big(N\times_{N^\gp}(R^\gp \oplus_TW)\big)\oplus_{\alpha_C^{-1}(C^*)}C^*=P^\llog$. We have $\frac{\pi^\flat(r_1)}{\pi^\flat(r_2)}=\frac{n}{1}$, $\pi^\sharp(f)\in B^*$. Since $(B,N)$ is a log ring, there exists $m\in N$ such that $\alpha_B(m)=\pi^\sharp(f)$. Let $\widetilde{m}\in R$ be such that $\pi^\flat(\widetilde{m})=m$. We have $\alpha_F(\widetilde{m})-f\in I$ and so $(f,x)=(\alpha_F(\widetilde{m}),x')$ where $x'=x-h_1'(\alpha_F(\widetilde{m})-f)$. Let $g\in C$ be such that $\pi^\sharp(g)=\tau^\sharp (f)^{-1}\in B^*$. We have
	\[ \alpha_C\big(m,\big(\frac{\widetilde{m}}{1},-gx'\big)\big)=\big(\alpha_F\big(\widetilde{m}\big),h_2\big(\widetilde{m}\otimes\frac{\widetilde{m}}{\widetilde{m}}\big)+mgx'\big)=(\alpha_F(\widetilde{m}),x')=(f,x') . \]
	Therefore
	\[ \big[(n,\big(\frac{r_1}{r_2},w\big)\big),(f,x)\big]=\big[\big(n,\big(\frac{r_1}{r_2},w\big)\big)\cdot \big(m,\big(\frac{\widetilde{m}}{1},-gx'\big)\big),(1,0)\big]=\big[\big(nm,\big(\frac{\widetilde{m}r_1}{r_2},w-gx'\big)\big),(1,0)\big] . \]
	So we can write any element of $P^\llog$ as $[(n,(\frac{r_1}{r_2},w)),(1,0)]$. An element $1+z\in 1+W$ acts on $P^\llog$ as the multiplication by $[(1,(\frac{1}{1},-z)),(1,0)]$ and so the action of $1+W$ on $P^\llog$ is free, since $(\frac{r_1}{r_2},w)=(\frac{r_1}{r_2},w-z)\in R^\gp\oplus_TW$ implies $z=0$, the map $T\to R^\gp$ being injective.

	If $h\in\theta(D,\widetilde{D})$, then $h_2'\colon T\to W$ factorizes as
	\[
	\begin{tikzcd}[column sep=4em,row sep=4em]
	T \arrowr & R^\gp \arrowr \arrow[bend right=20]{rr}{h_2''} & R^\gp/\im(M^\gp\to R^\gp) \arrowr{\widetilde{D}} & W
	\end{tikzcd}
	\]
	and then the exact sequence
	\[0 \longrightarrow W \longrightarrow R^\gp\oplus_T W \overset{\varepsilon}{\longrightarrow} N^\gp \longrightarrow 1 \]
	splits by $s\colon R^\gp\oplus_T W\to W$, $s(t,w)=h_2''(t)+w$. Therefore $R^\gp\oplus_T W=N^\gp\times W$ and so $P=N\times W$. Similarly, we can show that $C=B\oplus W$ and then the extension obtained is the trivial one.
	
	Finally, the two maps defined between classes of extensions and $H^1((A,M),(B,N),W)$ are inverse, and natural on $(B,N)$.
\end{proof}

\begin{lemma}\label{5.11}
	Let $B$ a ring, $J$ an ideal of $B$, $t>0$ an integer, $\{F_n\colon B/J^t-\text{modules} \to B/J^t-\text{modules} \}$ a direct system of half-exact functors (i.e. if $0\to U\to W\to V\to 0$ is exact, then $F_n(U)\to F_n(W)\to F_n(V)$ is exact). The following are equivalent:
	\begin{enumerate}
		\item[(i)] $\varinjlim\limits_n F_n(W)=0$ for all $B/J$-modules $W$ (considering $W$ as a $B/J^t$-module via $B/J^t\to B/J$),
		\item[(ii)] $\varinjlim\limits_n F_n(W)=0$ for all $B/J^t$-modules $W$.
	\end{enumerate}
\end{lemma}
\begin{proof}
	We have to show (i) $\Rightarrow$ (ii). Induction on $t$. The case $t=1$ is trivial. Let $t>1$, and assume the result valid for $t-1$. Let $W$ be a $B/J^t$-module. In the exact sequence
	\[ 0 \longrightarrow JW \longrightarrow W \longrightarrow W/JW \longrightarrow 0\]
	we have $\varinjlim\limits_n F_n(JW)=0$ since $JW$ is a $B/J^{t-1}$-module, and similarly $\varinjlim\limits_n F_n(W/JW)=0$. Since $\varinjlim\limits_n$ and $F_n$ are half-exact, we deduce $\varinjlim\limits_n F_n(W)=0$.
\end{proof}

\begin{theorem}\label{5.12}
	Let $(A,M) \to (B,N)$ be a homomorphism of prelog rings and $J$ an ideal of $B$. The following are equivalent:
	\begin{enumerate}
		\item[(i)] For each integer $n\geq 1$, let $p_n\colon B\to B/ J^n$ be the canonical map. Then \[\varinjlim\limits_n H^1((A,M),(B/ J^n,p_n^*N),W)=0\] for all $B/ J$-modules $W$.
		\item[(ii)] $(B,N)$ is log formally smooth over $(A,M)$ for the $J$-adic topology.
	\end{enumerate}
\end{theorem}
\begin{proof}
	(i) $\Rightarrow$ (ii) Consider a diagram as in Definition~\ref{5.03}.
	\[
	\begin{tikzcd}[column sep=4em,row sep=4em]
	(A,M) \arrowr\arrowd & (B,N) \arrowd{f} \\
	(C,P) \arrowr{\tau} & (C',P')
	\end{tikzcd}
	\]
	and let $I=\ker\tau^\sharp(C\to C')$. For all $n\geq t$ we have a commutative square
	\[
	\begin{tikzcd}[column sep=3em,row sep=4em]
	(D_n,Q_n) \arrowr{\theta_n}\arrowd{g_n} & (B/J^n,p_n^*N) \arrowd{f_n} \\
	(C,P) \arrowr{\tau} & (C',P')
	\end{tikzcd}
	\]
	where $p_n\colon B\to B/J^n$ is the canonical projection, $f_n$ is the map induced by $f$, $D_n=C\times_{C'}B/J^n$ and $Q_n=P\times_{P'}p_n^*N$.
	
	Both lines are $(A,M)$-extensions by $I$ (see Definition~\ref{5.07}). Lemma~\ref{5.11}.(i) tell us that
	\[\varinjlim_n  H^1((A,M),(B/J^n,p_n^*N),I)=0 , \]
	and so by Theorem~\ref{5.10} there exists some $s\geq t$ such that $\theta_s$ has a section $\sigma$. The map
	\[(B,N)\longrightarrow (B/J^s,p_s^*N) \overset{\sigma}{\longrightarrow} (D_s,Q_s) \overset{g_s}{\longrightarrow} (C,P)\]
	has the required properties of (ii).
	
	(ii) $\Rightarrow$ (i) Let  $\varphi_n\colon (D,Q) \to (B/J^n,p_n^*N)$ be an $(A,M)$-extension of $(B/J^n,p_n^*N)$ by $W$. By (ii), there exists a homomorphism of prelog rings $\phi_n\colon (B,N)\to (D,Q)$ such that $\varphi_n\phi_n=p_n$. Since $W^2=0$ as ideal of $D$, we have $\phi_n(J^{2n})=0$ and so $\phi_n$ factorizes as
	\[ (B,N) \overset{p_{2n}}{\longrightarrow} (B/J^{2n},p_{2n}^*N) \overset{\xi_{2n}}{\longrightarrow} (D,Q) .\]
	
	We have a commutative diagram where the rows are $(A,M)$-extensions by $W$
	\[
	\begin{tikzcd}[column sep=3em,row sep=4em]
	(D\times_{B/J^n}B/J^{2n},Q\times_{p_n^*N}p_{2n}^*N) \arrowr{\theta_{2n}}\arrowd & (B/J^{2n},p_{2n}^*N) \arrowd\arrow{ld}{\xi_{2n}} \\
	(D,Q) \arrowr & (B/J^{n},p_{n}^*N)
	\end{tikzcd}
	\]
	The homomorphism $\xi_{2n}$ induces a section of $\theta_{2n}$, showing that the element of $H^1((A,M),(B/J^n,p_n^*N),W)$ corresponding to the extension $\varphi_n$ (se Theorem~\ref{5.10}) goes to 0 in $H^1((A,M),(B/J^{2n},p_{2n}^*N),W)$. Note that if $(E,R)\overset{\theta}{\longrightarrow}(F,S)$ is an $(A,M)$-extension of $(F,S)$ by $W$ and $\sigma$ a section of $\theta$, then the map $S\times W\to R$, $(s,w)\mapsto (1+w)\sigma(s)$ is an isomorphism with inverse $r\mapsto(\theta(r),w')$ where $w'\in W$ is such that $(1+w')\sigma\theta(r)=r$ (there exists by Proposition~\ref{5.05}).	
\end{proof}

\begin{lemma}\label{5.13}
	Let $(B,N)$ be a prelog ring, $ J$ an ideal of $B$, $p_n\colon B\to B/ J^n$ the canonical homomorphism and $W$ a $B/J$-module.
	\begin{enumerate}
		\item[(i)] Assume that $J=0$ or $N$ is integral, then
		\[ \varinjlim_n H^i((B,N),(B/ J^n,p_n^*N),W)=0 \]
		for $i=0$ and $i=1$. 
		\item[(ii)] If $J=0$, or $B$ is a noetherian ring and $N$ integral, then
		\[ \varinjlim_n H^i((B,N),(B/ J^n,p_n^*N),W)=0 \]
		for all $i\geq 0$.
	\end{enumerate}
\end{lemma}
\begin{proof}
	If $J\neq 0$, since $(B/ J^n,p_n^*N)=(B/ J^n,N)^\llog$, we have by Proposition~\ref{4.14} and Example~\ref{4.02}
	\begin{align*}
	H^i((B,N),(B/ J^n,p_n^*N),W)&=H^i((B,N),(B/ J^n,N)^\llog,W)\\
	&=H^i((B,N),(B/ J^n,N),W)\\
	&=H^i(B,B/ J^n,W)
	\end{align*}
	and so the result follows from \cite[proof of 10.7 and of 10.14]{An-1974}.
\end{proof}

\begin{theorem}\label{5.14}
	Let $(A,M)\to (B,N)$ be a homomorphism of prelog rings, $ J$ an ideal of $B$. Assume $ J=0$, or $B$ noetherian and $N$ integral. The following are equivalent:
	\begin{enumerate}
		\item[(i)] $(B,N)$ is an $(A,M)$-algebra log formally smooth for the $ J$-adic topology.
		\item[(ii)] $H^1((A,M),(B,N),W)=0$ for all $B/ J$-modules $W$.
		\item[(iii)] $H_1((A,M),(B,N),B/ J)=0$ and $\Omega_{(B,N)|(A,M)}\otimes_BB/ J$ is a projective $B/ J$-module.
		\item[(iv)] For any homomorphism $(F,R)\to(B,N)$ of prelog rings where $(F,R)$ is a free $(A,M)$-prelog algebra and the homomorphisms $F\to B$ and $R\to N$ are surjective, the homomorphism
		\[ N_{(B,N)|(F,R)}\otimes_B B/ J \longrightarrow \Omega_{(F,R)|(A,M)}\otimes_F B/ J \]
		is split injective.
	\end{enumerate}
\end{theorem}
\begin{proof}
	(i) $\Leftrightarrow$ (ii) We have an exact sequence induced by Jacobi-Zariski exact sequences
	\begin{align*}
	& \varinjlim_n H^1((B,N),(B/ J^n,p_n^*N),W) \longrightarrow \varinjlim_n H^1((A,M),(B/ J^n,p_n^*N),W) \longrightarrow  \\
	\longrightarrow & \varinjlim_n H^1((A,M),(B,N),W) \longrightarrow \varinjlim_n H^2((B,N),(B/ J^n,p_n^*N),W) .
	\end{align*}
	Then, Lemma~\ref{5.13} gives an isomorphism
	\[\varinjlim_n H^1((A,M),(B/ J^n,p_n^*N),W)=H^1((A,M),(B,N),W) . \]
	So the result follows from Theorem~\ref{5.12}.
	
	(ii) $\Leftrightarrow$ (iii) The universal coefficient spectral sequence
	\begin{align*}
	E_2^{pq}=\Ext_{B/ J}^p(H_q( \LLL_{(B,N)|(A,M)}\otimes_BB/ J),W) \Rightarrow & \ H^{p+q}(\Hom_{B/ J}( \LLL_{(B,N)|(A,M)}\otimes_BB/ J, W)) \\
	& \quad = H^{p+q}((A,M),(B,N), W)
	\end{align*}	
	gives an exact sequence
	\[ 0\longrightarrow \Ext_{B/ J}^1(\Omega_{(B,N)|(A,M)}\otimes_BB/ J,W) \longrightarrow H^{1}((A,M),(B,N), W) \longrightarrow \]
	\[ \longrightarrow \Hom_{B/ J}(H_{1}((A,M),(B,N),B/J), W) \longrightarrow \Ext_{B/ J}^2(\Omega_{(B,N)|(A,M)}\otimes_BB/ J,W) , \]
	from which we obtain the desired equivalence.
	
	(iii) $\Leftrightarrow$ (iv) We have a Jacobi-Zariski exact sequence, taking by Proposition~\ref{4.06} the form
	\[ H_{1}((A,M),(F,R), B/ J) \longrightarrow H_{1}((A,M),(B,N), B/ J) \longrightarrow N_{(B,N)|(F,R)}\otimes_B B/ J \longrightarrow \]
	\[ \longrightarrow \Omega_{(F,R)|(A,M)}\otimes_F B/ J \longrightarrow \Omega_{(B,N)|(A,M)}\otimes_B B/ J \longrightarrow \Omega_{(B,N)|(F,R)}\otimes_B B/ J=0 .\]
	The first module vanishes by Proposition~\ref{4.07} 
	and furthermore $\Omega_{(F,R)|(A,M)}\otimes_F B/ J$ is a free $B/ J$ module by Example~\ref{3.14}.(i). Then the result follows.
\end{proof}

\begin{theorem}\label{5.15}
	Let $(A,M)\to (B,N)$ be a homomorphism essentially of finite type of noetherian prelog rings (see Definition~\ref{2.01}) and $J$ an ideal of $B$. Assume $J=0$ or $N$ is integral. The following are equivalent:
	\begin{enumerate}
		\item[(i)] $(B,N)$ is an $(A,M)$-algebra log formally smooth for the $ J$-adic topology.
		\item[(ii)]  $H_1((A,M),(B,N),W)=0$ for all $B/ J$-modules $W$.
	\end{enumerate}
	If moreover $B$ is local and $ J$ a proper ideal of $B$, these conditions are also equivalent to
	\begin{enumerate}
		\item[(iii)] $(B,N)$ is an $(A,M)$-algebra log formally smooth for the $0$-adic topology.
	\end{enumerate}
\end{theorem}
\begin{proof}
	Conditions (i) and (ii) are equivalent by Theorem~\ref{5.14} and Proposition~\ref{4.05}.
	
	(iii) $\Rightarrow$ (ii) is clear by Theorem~\ref{5.14}.(ii) or by definition.
	
	Assume (i). Then by Theorem~\ref{5.14}.(ii) we have $H^1((A,M),(B,N),B/\nnn)=0$ where $\nnn$ is the maximal ideal of $B$, and so by Proposition~\ref{4.05} we deduce $H^1((A,M),(B,N),W)=0$ for any $B$-module $W$. So we get (iii) by Theorem~\ref{5.14}.
\end{proof}

We will see some examples where homological methods give short proofs of known results. If $C$ is a finitely generated abelian group and $R$ is a ring, then $C$ is finite of order inversible in $R$ if and only if $\Hom_{\Z}(C,W)=0$ for any $R$-module $W$, and the torsion part of $C$ is of order inversible in $R$ if and only if $\Ext_{\Z}^1(C,W)=0$ for any $R$-module $W$. So the following two results are included in \cite[Theorem IV.3.1.8, Proposition IV.3.1.13]{Og} in the case of finitely generated monoids (and taking $J=0$).

\begin{proposition}\label{5.16}
	Let $M\to N$ be a homomorphism of monoids with $N$ integral, $R$ a ring, $A=R[M]$, $B=R[N]$ and $J$ an ideal of $B$. If
	\[ \Hom_{\Z}(\ker(M^\gp\to N^\gp),W)=0=\Ext_{Z}^1(\coker(M^\gp\to N^\gp),W) \]
	for any $B/J$-module $W$, then $(A,M)\to(B,N)$ is log formally smooth for the $J$-adic topology.
\end{proposition}
\begin{proof}
	Let $W$ be a  $B/J$-module and consider the fundamental exact sequence
	\[ H^0(R[M],R[N],W)\oplus\Hom_{\Z}(\coker(M^\gp\to N^\gp),W)\longrightarrow H^0(\Z[M],\Z[N],W)\longrightarrow\]
	\[ \longrightarrow H^1((A,M),(B,N),W)\longrightarrow \]
	\[ \longrightarrow H^1(R[M],R[N],W)\oplus\Ext_{\Z}^1(\coker(M^\gp\to N^\gp),W)\oplus\Hom_{\Z}(\ker(M^\gp\to N^\gp),W) \longrightarrow\]
	\[ \longrightarrow H^1(\Z[M],\Z[N],W) . \]
	
	The canonical homomorphism $H^0(R[M],R[N],W)\to  H^0(\Z[M],\Z[N],W)$ is an isomorphism by propositions~\ref{3.10} and \ref{3.16}, and the homomorphism \[ H^1(R[M],R[N],W)\longrightarrow H^1(\Z[M],\Z[N],W) \] is injective by the cohomological analogue (same proof) of \cite[2.6.2]{MR}. The result follows from\linebreak Theorem~\ref{5.14}.
\end{proof}

\begin{proposition}\label{5.17}
	Let $\varphi\colon (A,M)\to(B,N)$ be a homomorphism of prelog rings and $J$ an ideal of $B$. Assume that $N$ is integral and
	\begin{align*}
	\Hom_{\Z}(\ker(M^\gp\to N^\gp),W)&=\Ext_\Z^1(\coker(M^\gp\to N^\gp),W) \\
	&=\Hom_\Z(\coker(M^\gp\to N^\gp), W)=0
	\end{align*}
	for any $B/J$-module $W$. The following are equivalent:
	\begin{enumerate}
		\item[(i)] $\varphi$ is log formally smooth for the $J$-adic topology.
		\item[(ii)] $B$ is formally smooth over $A\otimes_{\Z[M]}\Z[N]$ for the $J$-adic topology.
	\end{enumerate}
\end{proposition}
\begin{proof}
	From the fundamental sequence and Theorem~\ref{5.14}
	\[
	\begin{tikzcd}[column sep=2em,row sep=0.5em]
	& H^0(A,B,W) \arrowr{\beta_0} & H^0(\Z[M],\Z[N],W) \arrowr &  \ \\
	H^1((A,M),(B,N),W) \arrowr & H^1(A,B,W) \arrowr{\beta_1} & H^1(\Z[M],\Z[N],W) \\
	\end{tikzcd}
	\]
	we see that (i) holds if and only if $\beta_1$ is injective and $\beta_0$ is surjective for all $B/J$-modules $W$.
	
	Now, from the Jacobi-Zariski exact sequence
	\[
	\begin{tikzcd}[column sep=2em,row sep=0.5em]
	& H^0(A,B,W) \arrowr{\beta_0'} & H^0(A,A\otimes_{Z[M]}\Z[N],W) \arrowr &  \ \\
	H^1(A\otimes_{Z[M]}\Z[N],B,W) \arrowr & H^1(A,B,W) \arrowr{\beta_1'} & H^1(A,A\otimes_{Z[M]}\Z[N],W) \\
	\end{tikzcd}
	\]
	we see that (ii) holds if and only if $\beta_1'$ is injective and $\beta_0$ is surjective.
	
	Since in the commutative triangle
	\[
	\begin{tikzcd}[column sep=0em,row sep=3em]
	H^i(A,B,W) \arrow{rr}{\beta_i}\arrow{rd}{\beta_i'} & & H^i(\Z[M],\Z[N],W) \\
	& H^i(A,A\otimes_{Z[M]}\Z[N],W) \arrow{ru}{\gamma_i} & 
	\end{tikzcd}
	\]
	$\gamma_1$ is injective by the cohomological analogue of \cite[2.6.2]{MR} and $\gamma_0$ is bijective, we are done.	
\end{proof}

Here is a version of \cite[Theorem 0.1]{INT}.

\begin{proposition}\label{5.18}
	Let $(A,M)\to (B,N)$ and $(A,M)\to(C,P)$ be homomorphisms of prelog rings. Then (i) $\Rightarrow$ (ii) $\Rightarrow$ (iii) $\Rightarrow$ (iv) $\Rightarrow$ (v) $\Rightarrow$ (vi), where:
	\begin{enumerate}
		\item[(i)] $M\to P$ is a Kummer homomorphism of integral monoids, $M$ is saturated and $A\otimes_{\Z[M]}\Z[P]\to C$ is flat.
		\item[(ii)] $M\to P$ is an exact and injective homomorphism of integral monoids and $A\otimes_{\Z[M]}\Z[P]\to C$ is flat.
		\item[(iii)] $\Z[M]\to\Z[P]$ and $A\otimes_{\Z[M]}\Z[P]\to C$ are flat, and $M^\gp\to P^\gp$ is injective.
		\item[(iv)] $\Z[M]\to\Z[P]$ and $A\to C$ are flat, and $M^\gp\to P^\gp$ is injective.
		\item[(v)] $\Tor_n^{\Z[M]}(\Z[N],\Z[P]))=0=\Tor_n^A(B,C)$ for all $n>0$, and $M^\gp\to P^\gp$ is injective.
		\item[(vi)] For any $B\otimes_AC$-module $W$, we have
		\[H_n((A,M),(B,N),W)=H_n((C,P),(B\otimes_AC,N\oplus_MP),W)\]
		for all $n\geq 0$. 
	\end{enumerate}
\end{proposition}
\begin{proof}
	(i) $\Rightarrow$ (ii) follows from the definition. (ii) $\Rightarrow$ (iii), \cite[Lemma 2.1]{INT}: Let $m\in M$ and $p\in P$. If $mp\in M$ then $p=(mp)/m \in M^\gp$ and then $p\in M$ since $M\to P$ is exact. So the homomorphism of $\Z[M]$-modules $\Z[M]\to \Z[P]$ is split by $\Z[P]\to\Z[M]$, $p\mapsto0$, if $p\in P-M$.
	
	On the other hand, if $\Z[M]\to\Z[P]$ is flat then $A\to A\otimes_{\Z[M]}\Z[P]$ is flat. If moreover $A\otimes_{\Z[M]}\Z[P]\to C$ is flat, then by composition $A\to C$ is flat too. This shows (iii) $\Rightarrow$ (iv).
	
	Finally, (iv) $\Rightarrow$ (v) is trivial and (v) $\Rightarrow$ (vi) is Proposition~\ref{4.08}.
\end{proof}

\begin{corollary}\label{5.19}
	Let $(A,M)\to (B,N)$ and $(A,M)\to(C,P)$ be homomorphisms of prelog rings with $\Spec(C)\to\Spec(A)$ surjective and assume that condition (iv) of Proposition~\ref{5.18} holds. Let $J$ be an ideal of $B$. Then $(B,N)$ is an $(A,M)$-algebra formally smooth for the $J$-adic topology if and only if $(B\otimes_AC,N\oplus_MP)$ is a $(C,P)$-algebra formally smooth for the $J\otimes_AC$-adic topology.
\end{corollary}
\begin{proof}
	We know that $A\to C$ is a faithfully flat homomorphism, so for any $B/J$-module $W$ we have that $H_1((A,M),(B,N),W)=0$ if and only if \[H_1((A,M),(B,N),W)\otimes_AC=H_1((A,M),(B,N),W\otimes_AC)=0 . \] Then the result follows from Proposition~\ref{5.18}, (iv) $\Rightarrow$ (vi).
\end{proof}

\section{Regularity}\label{R}
\begin{theorem}\label{6.01}
	Let $((A,\mmm,k),M)$ be a noetherian local prelog ring, $J$ a proper ideal of $A$ and $B:=A/J$. Let $N:=M/\alpha_M^{-1}(J)$ be the quotient of the monoid $M$ by the ideal $\alpha_M^{-1}(J)$, and so $(B,N)$ is a noetherian local prelog ring with residue field $k$. Let $I$ be the ideal of $A$ generated by $\alpha(M)\cap J$. The following statements are equivalent:
	\begin{enumerate}
		\item[(i)] For any $B$-module $W$, the map
		\[ H_2(A,B,W)\oplus\Tor_1^\Z(\ker(M^\gp\to N^\gp), W)\longrightarrow H_2((A,M),(B,N),W) \]
		in the fundamental exact sequence is zero on the first summand and an isomorphism on the second one.
		
		\item[(i')] The map
		\[ H_2(A,B,k)\oplus\Tor_1^\Z(\ker(M^\gp\to N^\gp), k)\longrightarrow H_2((A,M),(B,N),k) \]
		in the fundamental exact sequence is zero on the first summand and an isomorphism on the second one.
		
		\item[(ii)] $\ker(A/I\to B)$ is generated by a regular sequence and $\Tor_1^{\Z[M]}(A,\Z[N])=0$.
	\end{enumerate}
\end{theorem}
\begin{proof}
	First note that $\ker(\Z[M]\to\Z[N])$ is the ideal $\aaa$ of $\Z[M]$ generated by the elements $a-b$ with $a,b\in\alpha_M^{-1}(J)$, and in particular $\aaa A\subset I$. If $x\in \alpha(M)\cap J \subset \mmm$, $1-x$ is a unit in $A$, and so $x(1-x)=x-x^2\in\aaa A$ implies $x\in\aaa A$. Therefore $\aaa A=I$. Similarly,
	$\Tor_*^{\Z[M]}(A,\Z[N])=\Tor_*^{\Z[M]}(A,\Z[M]/I')$ where $I'$ is the ideal of $\Z[M]$ generated by $\alpha_M^{-1}(J)$, since for $a\in\alpha_M^{-1}(J)$ the localization
	\[ \Z[N]_{1-a} = \Z[M]_{1-a}/\aaa\Z[M]_{1-a} = \Z[M]_{1-a}/I'\Z[M]_{1-a} =\bigg( \Z[M]/I'\bigg)_{1-a} \]
	as above, and the image of $1-a$ in $A$ being a unit,
	\[ \Tor_*^{\Z[M]}(A,X)=\Tor_*^{\Z[M]}(A,X)_{1-a}=\Tor_*^{\Z[M]}(A,X_{1-a}) .\]
	
	(i') $\Rightarrow$ (ii) Consider the following diagram with exact rows (Theorem~\ref{4.03}) and columns (Jacobi-Zariski exact sequence)
	\[
	\hspace{-0.02\textwidth}
	\begin{tikzpicture}[baseline= (a).base]
	\node[scale=0.85] (a) at (0,0){
		\begin{tikzcd}[row sep=2em, column sep=1em]
		& H_3(A/ I,B,k) \arrow[equal]{r} \arrowd & H_3(A/ I,B,k) \arrowd & & \\
		H_2(\Z[M],\Z[N],k) \arrowr{\beta_2}\arrow[equal]{d}  &  H_2(A,A/ I,k)\oplus H_2(X) \arrowr\arrowd{\alpha_2\oplus \id} & H_2((A,M),(A/ I,N),k) \arrowr\arrowd & H_1(\Z[M],\Z[N],k) \arrowr{\beta_1}\arrow[equal]{d} & H_1(A,A/I,k)\oplus H_1(X) \arrowd{\alpha_1\oplus\id} \\
		H_2(\Z[M],\Z[N],k) \arrowr{\gamma_2} & H_2(A,B,k)\oplus H_2(X) \arrowr\arrowd{\mu} & H_2((A,M),(B,N),k) \arrowr\arrowd & H_1(\Z[M],\Z[N],k) \arrowr{\gamma_1} & H_1(A,B,k)\oplus H_1(X) \\
		& H_2(A/ I,B,k) \arrow[equal]{r} \arrowd{\lambda} & H_2(A/ I,B,k) & & \\
		& H_1(A,A/ I,k) \arrowd{\alpha_1} & & & \\
		& H_1(A,B,k) & & &
		\end{tikzcd}
	};
	\end{tikzpicture}\]
	where $H_2(X)=\Tor_1^{\Z}(\ker(M^\gp \to N^\gp),k)$ and $H_1(X)=\ker(M^\gp \to N^\gp)\otimes_\Z k$ as in Theorem~\ref{4.03}. If $H_2(X)\to H_2((A,M),(B,N),k)$ is an isomorphism, then $\gamma_1$ is injective. Moreover, the composition
	\[ H_1(\Z[M],\Z[N],k) \overset{\gamma_1}{\longrightarrow} H_1(A,B,k)\oplus H_1(X) \overset{\pr_2}{\longrightarrow} H_1(X) \]
	is zero since it is induced by the map of Proposition~\ref{3.01}
	\[\alpha_M^{-1}(J)\Z[M]/\alpha_J^{-1}(J)^2\Z[M]\otimes_{\Z[N]}k \longrightarrow  \ker(M^\gp \to N^\gp)\otimes_\Z k \]
	defined by
	\[ \overline{\sum_{i}\lambda_i m_i} \otimes w \longmapsto \sum_i m_i \otimes \lambda_i\widetilde{m}_iw=0 ,\]
	where $\widetilde{m}_i$ is the image of $m_i\in B$ which is zero since $m_i\in\alpha_M^{-1}(J)$.
	
	Therefore $\gamma_1$ injective and $\pr_2\gamma_1=0$ imply that
	\[
	\begin{tikzcd}[column sep=3em,row sep=1em]
	H_1(\Z[M],\Z[N],k) \arrowr{\pr_1\gamma_1} & H_1(\Z[M],\Z[N],k)
	\end{tikzcd}
	\]
	is injective too. By the commutativity of the diagram, the homomorphism 
	\[
	\begin{tikzcd}[column sep=3em,row sep=1em]
	H_1(\Z[M],\Z[N],k) \arrowr{\pr_1\beta_1} & H_1(A,A/I,k)
	\end{tikzcd}
	\]
	is injective. It is also an isomorphism, since it is surjective by \cite[2.6.2]{MR} (since $A\otimes_{\Z[M]}\Z[N]=A/\aaa A=A/I$), and then, again by the commutativity of the diagram, $\alpha_1$ is injective. We deduce $\lambda=0$. Since \[ H_2(A,B,k)\longrightarrow H_2((A,M),(B,N),k) \]
	is the zero map, we deduce $\mu=0$. Therefore
	\[ H_2(A/I,B,k)=0\]
	which by \cite[6.25]{An-1974} means that  $\ker(A/I\to B)$ is generated by a regular sequence.
	
	Now we will see that $\Tor_1^{\Z[M]}(A,\Z[N])=0$. We have $H_2(A/I,B,k)=0$ and then $H_3(A/I,B,k)=0$ by \cite[6.25]{An-1974}. Therefore in the diagram
	\[
	\begin{tikzcd}[column sep=3em,row sep=4em]
	H_2(\Z[M],\Z[N],k) \arrowr{\pr_1\beta_2}\arrow[equal]{d} & H_2(A,A/I,k) \arrowd{\alpha_2} \\
	H_2(\Z[M],\Z[N],k) \arrowr{\pr_1\gamma_2} & H_2(A,B,k)
	\end{tikzcd}
	\]
	$\alpha_2$ is an isomorphism. Since the lower map is surjective by hypothesis, the upper one is also surjective. From the exact sequence \cite[15.18]{An-1974}
	\[
	\begin{tikzcd}[column sep=3em,row sep=1em]
	& H_2(\Z[M],\Z[N],k) \arrowr{\pr_1\beta_2} & H_2(A,A/I,k) \arrowr & \Tor_1^{\Z[M]}(A,\Z[N])\otimes_Ak \arrowr & \ \\
	\arrowr & H_1(\Z[M],\Z[N],k) \arrowr{\pr_1\beta_1} & H_1(A,A/I,k) & 
	\end{tikzcd}
	\]
	we deduce $\Tor_1^{\Z[M]}(A,\Z[N])\otimes_Ak=0$ and then, since $A$ is noetherian local, $\Tor_1^{\Z[M]}(A,\Z[N])=0$.
	
	(ii) $\Rightarrow$ (i) By \cite[15.18]{An-1974} we have an exact sequence for each $B$-module $W$
	\[
	\begin{tikzcd}[column sep=3em,row sep=1em]
	& H_2(\Z[M],\Z[N],W) \arrowr{\theta_2} & H_2(A,A/I,W) \arrowr & \Tor_1^{\Z[M]}(A,\Z[N])\otimes_AW \arrowr & \ \\
	\arrowr & H_1(\Z[M],\Z[N],W) \arrowr{\theta_1} & H_1(A,A/I,W) \arrowr & 0  .
	\end{tikzcd}
	\]
	By hypothesis, $\Tor_1^{\Z[M]}(A,\Z[N])=0$, so $\theta_2$ is surjective and $\theta_1$ injective. Since $H_n(A/I,B,W)=0$ for all $n\geq 2$, we deduce that in the diagram
	\[
	\begin{tikzcd}[column sep=3em,row sep=4em]
	H_2(\Z[M],\Z[N],W) \arrowr{\theta_2}\arrow[equal]{d} & H_2(A,A/I,W) \arrowd{\alpha_2} \\
	H_2(\Z[M],\Z[N],W) \arrowr{\psi_2} & H_2(A,B,W)
	\end{tikzcd}
	\]
	the map $\alpha_2$ is bijective and so $\psi_2$ is surjective.
	
	In the diagram
	\[
	\begin{tikzcd}[column sep=3em,row sep=4em]
	& H_2(A/I,B,W)=0 \arrowd \\
	H_1(\Z[M],\Z[N],W) \arrowr{\theta_1}\arrow[equal]{d} & H_1(A,A/I,W) \arrowd{\alpha_1} \\
	H_1(\Z[M],\Z[N],W) \arrowr{\psi_1} & H_1(A,B,W)
	\end{tikzcd}
	\]
	$\theta_1$ is injective and then $\psi_1$ is also injective.
	
	From the fundamental exact sequence
	\[ H_2(\Z[M],\Z[N],W) \overset{\psi_2\oplus\omega_2}{\longrightarrow} H_2(A,B,W)\oplus\Tor_1^{Z[M]}(\ker(M^\gp\to N^\gp),W) \longrightarrow H_2((A,M),(N,N),W) \longrightarrow \]
	\[\longrightarrow H_1(\Z[M],\Z[N],W) \overset{\psi_1\oplus\omega_1}{\longrightarrow} H_1(A,B,W)\oplus\ker(M^\gp\to N^\gp)\otimes_\Z W \]
	we see that we are done if we prove that the map
	\[ H_2(\Z[M],\Z[N ],W) \longrightarrow \Tor_1^{Z[M]}(\ker(M^\gp\to N^\gp),W) \]
	vanishes.
	
	We have a surjective natural homomorphism \cite[Theorem 6.16]{Quillen-MIT}
	\[ \Tor_2^{\Z[M]}(\Z[N],W)\longrightarrow H_2(\Z[M],\Z[N],W) , \]
	so it is sufficient to show that the composition
	\[ \Tor_2^{\Z[M]}(\Z[N],W)\longrightarrow H_2(\Z[M],\Z[N],W)\longrightarrow\Tor_1^{\Z}(\ker(M^\gp\to N^\gp),W)\]
	is zero. We have isomorphisms
	\[\Tor_1^{\Z[M]}(\aaa,W)=\Tor_2^{\Z[M]}(\Z[N],W) , \]
	\[\Tor_1^{\Z}(\ker(M^\gp\to N^\gp),W)=\Tor_1^{\Z[M]}(\Z[M]\otimes_\Z\ker(M^\gp\to N^\gp),W)\]
	where $\aaa=\ker(\Z[M]\to\Z[N])$. Since the map
	\[ \aaa/\aaa^2 \longrightarrow \Z[M]\otimes_\Z\ker(M^\gp \to N^\gp) \]
	of Proposition~\ref{3.01} vanishes after tensoring by $\otimes_{\Z[M]}B$ and $W$ is a  $B$-module, the map
	\[ \Tor_1^{\Z[M]}(\aaa,W)\longrightarrow\Tor_1^{Z[M]}(\Z[M]\otimes_\Z\ker(M^\gp\to N^\gp),W)\]
	is zero as desired.
\end{proof}

\begin{definition}(see \cite[Theorem 6.1]{Ka-TS})\label{6.02}
	Let $((A,\mmm,k),M)$ be a noetherian local prelog ring, $J$ a proper ideal of $A$ and $B=A/J$. Let $N=M/\alpha_M^{-1}(J)$ be the quotient of the monoid $M$ by the ideal $\alpha_M^{-1}(J)$ and let $I$ be the ideal of $A$ generated by $\alpha(M)\cap J$. We say that $J$ is a \emph{log regular ideal} if the equivalent conditions of Theorem~\ref{6.01} hold.
	
	We say that $((A,\mmm,k),M)$ is a \emph{log regular local ring} if $\mmm$ is a log regular ideal.
\end{definition}

\begin{theorem}\label{6.03}
	Let $((A,\mmm,k),M)$ be a noetherian local prelog ring, and $((A,\mmm,k),M')$ another prelog structure inducing the same log structure. Assume that $M$ and $M'$ are integral. For any proper ideal $J$ of $A$, $J$ is log regular in $((A,\mmm,k),M)$ if and only if it is log regular in $((A,\mmm,k),M')$.
\end{theorem}
\begin{proof}
	It is sufficient to show the result for $M'=M^\llog$. With the notation of Definition~\ref{6.02}, we have $I:=<\alpha_M(M)\cap J>_A=<\alpha_{M^\llog}(M^\llog)\cap J)>$ since for any element $x=(m,u)\in M\oplus_{\alpha_M^{-1}(A^*)}A^*=M^\llog$, we have $\alpha_{M^\llog}(x)=\alpha(m)\cdot u$ where $u$ is a unit in $A$. In particular, $\ker(A/I\to B)$ do not change replacing $M$ by $M^\llog$.
	
	We also have
	\[ \Z\big[M^\llog / \alpha_{M^\llog}^{-1}(J)\big]=
	\Z\big[M / \alpha_{M}^{-1}(J)\big]\otimes_{\Z[M]}\Z[M^\llog]=\Z[N]\otimes_{\Z[M]}\Z[M^\llog] , \]
	so it is enough to show that
	\[ \Tor_1^{\Z[M]}(A,\Z[N]) = \Tor_1^{\Z[M^\llog]}(A,\Z[N]\otimes_{\Z[M]}\Z[M^\llog]) . \]
	For simplicity, let us denote $\alpha=\alpha_M\colon M\to A$, and factorize $\alpha^{-1}(A^*)\to A^*$ first as $\alpha^{-1}(A^*)\overset{u}{\longrightarrow} \alpha^{-1}(A^*)^\gp \overset{v}{\longrightarrow} A^*$ and factorize again the second map: $\alpha^{-1}(A^*)\overset{u}{\longrightarrow} \alpha^{-1}(A^*)^\gp \overset{v_1}{\longrightarrow} \alpha^{-1}(A^*)^\gp/\ker v \overset{v_2}{\longrightarrow} A^*$.
	
	Therefore, the pushout that defines $M^\llog$
	\[
	\begin{tikzcd}[column sep=4em,row sep=4em]
	\alpha^{-1}(A^*) \arrowr\arrowd & A^* \arrowd \\
	M \arrowr & M^\llog
	\end{tikzcd}
	\]
	can be descompose in a commutative diagram of pushouts
	\[
	\begin{tikzcd}[column sep=3em,row sep=4em]
	\alpha^{-1}(A^*) \arrowr{u}\arrowd{r} & \alpha^{-1}(A^*)^\gp \arrowr{v_1}\arrowd{s} &
	\alpha^{-1}(A^*)^\gp/\ker(v) \arrowr{v_2}\arrowd & A^* \arrowd \\
	M \arrowr & M_1 \arrowr & M_2 \arrowr & M^\llog
	\end{tikzcd}
	\]
	Note that $s$ is injective since it is a localization of the injective homomorphism $r$ and $M$ is integral.
	
	Now, we denote as
	\[
	\begin{tikzcd}[column sep=4em,row sep=4em]
	N_1 \arrowr\arrowd & N_2 \arrowd \\
	N_3 \arrowr & N_4
	\end{tikzcd}
	\]
	any of previous pushouts and consider $W$ a $\Z[M^\llog]$-module and $P=N_3/T$ with $T$ an ideal of $N_3$. Let $N_3\to X\to P$ a factorization cofibration-trivial fibration in the category of simplicial monoids. Then $\Z[X]$ is a projective $\Z[N_3]$-resolution of $\Z[P]$ \cite[I.2.2.3]{Il-CC} and so
	\begin{align*}
	\Tor_*^{\Z[N_3]}(W,\Z[P])&=H_*(W\otimes_{\Z[N_3]}\Z[X])=H_*(W\otimes_{\Z[N_4]}\Z[N_4]\otimes_{\Z[N_3]}\Z[X]) \\
	&\overset{(1)}{=}\Tor_*^{\Z[N_4]}(W,\Z[N_4]\otimes_{\Z[N_3]}\Z[P])
	\end{align*}
	where the equality (1) follows from:
	\begin{align*}
	H_*(\Z[N_4]\otimes_{\Z[N_3]}\Z[X])&=H_*(\Z[N_3\oplus_{N_1}N_2]\otimes_{\Z[N_3]}\Z[X]) \\
	&=H_*(\Z[N_3]\otimes_{\Z[N_1]}\Z[N_2]\otimes_{\Z[N_3]}\Z[X]) \\
	&= H_*(\Z[N_2]\otimes_{\Z[N_1]}\Z[X]) \overset{(2)}{=} \Z[N_2]\otimes_{\Z[N_1]}\Z[P] \\
	&= \Z[N_4]\otimes_{\Z[N_3]}\Z[P]
	\end{align*}
	where the equality (2) follows (depending of the pushout) from:
	\begin{itemize}
		\item In the first pushout, since $\Z[N_2]$ is a flat $\Z[N_1]$ module (because $u$ is a localization and then $\Z[u]$ is a localization too).
		\item In the second one, it follows by Remark~\ref{4.11}, since $\Z[N_2]\otimes_{\Z[N_1]}\Z[X]=\Z[X/\ker(v)]$ is a resolution of $\Z[P/\ker(v)]=\Z[N_2]\otimes_{\Z[N_1]}\Z[P]$. Note that $P$ is $M/T\oplus_M M_1=M_1/\widetilde{T}$ in this case (where $\widetilde{T}$ denotes the ideal image of $T$ in $M_1$) and $\ker(v)\to \alpha^{-1}(A^*)^\gp \to M_1$ is injective. Moreover, since $\tilde{T}$ is an ideal of $M_1$, $\ker(v)\to M_1/\tilde{T}$ is injective. On the other hand, the action of $\ker(v)$ on $X$ is free, since
		\[X=M_1\oplus\N^{\widetilde{X}} ,\]
		and $M_1$ is integral (the homomorphism $u$ is integral by \cite[6.2.5.(i)]{GR} and so $M_1$ is integral because $M$ it is also integral).
		
		\item In the last pushout, it follows because $\Z[N_2]$ is a free $\Z[N_1]$ module (since $v_2$ is injective homomorphism of groups).
	\end{itemize}
	Finally, applying successively this process to the three pushouts, we obtain as desired
	\[
	\Tor_*^{\Z[M]}(A,\Z[N])=\Tor_*^{\Z[M^\llog]}(A,\Z[N]\otimes_{\Z[M]}\Z[M^\llog]) .	\]	
\end{proof}

\begin{corollary}\label{6.04}
	With the notation of Theorem~\ref{6.01}, assume that $M^\gp$ is a free abelian group. Then $J$ is log regular  if and only if $H_2((A,M),(B,N),k)=0$.
\end{corollary}

By Lemma~\ref{6.06}, we can always work in this case ($M^\gp$ free).

\begin{lemma}\label{6.05}
	Let $p\colon G_1\to G_2$ be a surjective homomorphism of abelian groups, $i:N\to G_2$ a homomorphism of monoids. Then $(G_1\times_{G_2}N)^\gp=G_1\times_{G_2}N^\gp$.
\end{lemma}
\begin{proof}
	The homomorphism
	\begin{align*}
	(G_1\times_{G_2}N)^\gp \;&\longrightarrow\; G_1\times_{G_2}N^\gp \\
	\frac{(g_1,n_1)}{(g_2,n_2)} \;&\longmapsto\; \bigg(\frac{g_1}{g_2},\frac{n_1}{n_2}\bigg)
	\end{align*}
	is well defined and has inverse
	\[\bigg(g,\frac{n_1}{n_2}\bigg) \;\longmapsto\; \frac{(gh,n_1)}{(h,n_2)}\]	where $h\in p^{-1}(i(n_2))$.
\end{proof}

\begin{lemma}\label{6.06}
	Let $(B,N)$ be a log ring. There exists a monoid $P$ and a homomorphism of monoids $P\to N$ such that $(B,N)=(B,P)^\llog$ and $P^\gp$ is a free abelian group.
	
	Moreover, if $N$ is integral, then $P$ can be chosen integral too.
\end{lemma}
\begin{proof} \cite[12.1.35]{GR}
	Let $G$ be a free abelian group with a surjective homomorphism of groups $\rho\colon G\to N^\gp$. Let $P:=G\times_{N^\gp}N$. We have a commutative diagram of pullbacks
	\[
	\begin{tikzcd}[row sep=4em, column sep=3em]
	H:=(\beta\tau)^{-1}(B^*) \arrowr{\pi}\arrowd & B^*\arrowd \\
	P \arrowr{\tau}\arrowd{j} & N \arrowd{i} \\
	G \arrowr{\rho} & N^\gp
	\end{tikzcd}
	\]
	where $\beta\colon N\to B$ is the structural homomorphism. So $\ker \pi \; (:=\pi^{-1}(\{1\}))=\ker\tau=\ker\rho$ and $\ker\rho$ is a group. If $P/\ker(\tau)$ denotes the quotient by the action of $\ker(\tau)$ on $P$, it is easy to check that $P/\ker(\tau)=N$ and $H/\ker(\pi)=B^*$.
	
	Moreover,
	$$P\oplus_HB^*=P/\ker(\tau)\oplus_{H/\ker(\pi)}B^*=N\oplus_{B^*}B^*=N$$
	and so $(B,P)^\llog=(B,N)$. Finally, $P^\gp=G$ by Lemma~\ref{6.05}, and so it is a free abelian group.
	
	Finally, if $N$ is integral, that is, $i$ is injective, then $j$ is injective, and therefore $P$ is integral.
\end{proof}

\begin{lemma}\label{6.07}
	Let $((A,\mmm,k),M)$ be a noetherian local prelog ring and $\mmm_M=M-M^*$ the maximal ideal of $M$. Assume that $M$ is integral. If $\Tor_1^{\Z[M]}(A,\Z[M/\mmm_M])=0$, then $\Tor_1^{\Z[M]}(A,\Z[M/ I])=0$ for any ideal $ I$ of $M$.
\end{lemma}
\begin{proof}
	Since $M$ is noetherian, there exists a maximal element of the set of ideals of $M$ such that $\Tor_1^{\Z[M]}(A,\Z[M/ I])\neq0$, assuming this set is not empty. Clearly, $ I\subset \mmm_M$ (if not $ I=M$), and then by hypothesis, $I\subsetneq \mmm_M$. Let $a\in\mmm_M- I$. Let $T=\{x\in M / xa\in  I\}$, so that we have an exact sequence
	\[ 0\longrightarrow \Z[M/T]\overset{\cdot a}{\longrightarrow}\Z[M/ I]\longrightarrow\Z[M/ I\cup Ma]\longrightarrow 0 . \]
	We have $ I\subset T$. If $ I\subsetneq T$, $\Tor_1^{\Z[M]}(A,\Z[M/T])=0=\Tor_1^{\Z[M]}(A,\Z[M/ I\cup Ma])$ by the maximality of $I$. So $\Tor_1^{\Z[M]}(A,\Z[M/ I])=0$.
	
	If $ I=T$, we have a surjective homomorphism of $A$-modules of finite type
	\[
	\Tor_1^{\Z[M]}(A,\Z[M/ I])\overset{\cdot\tilde{a}}{\longrightarrow}\Tor_1^{\Z[M]}(A,\Z[M/ I])
	\]
	where $\tilde{a}:=\alpha(a)\in\mmm\subset A$. By Nakayama Lemma, $\Tor_1^{\Z[M]}(A,\Z[M/ I])=0$.	
\end{proof}

\begin{corollary}\label{6.08}
	Let $((A,\mmm,k),M)$ be a noetherian local prelog ring and $\mmm_M=M-M^*$ the maximal ideal of $M$. Assume that $M$ is integral. If  $\Tor_1^{\Z[M]}(A,\Z[M/\mmm_M])=0$, then
	\[\Tor_n^{\Z[M]}(A,\Z[M/\mmm_M])=0\]
	for all $n\geq 1$.
\end{corollary}
\begin{proof}
	By Lemma~\ref{6.07}, it is sufficient to show that if $\Tor_n^{\Z[M]}(A,\Z[M/ I])$ for any ideal $ I$ of $M$ for some $n\geq 1$, then $\Tor_{n+1}^{\Z[M]}(A,\Z[M/ I])=0$ for any ideal $ I$ of $M$. So assume $\Tor_n^{\Z[M]}(A,\Z[M/ I])=0$ for any ideal $ I$ of $M$. By \cite[6.1.27(i)-(iii)]{GR}, we have then $\Tor_n^{\Z[M]}(A,\Z[T])=0$ for any $M$-set $T$. Applying $\Tor$ to the exact sequence
	\[
	0 \longrightarrow\Z[I]\longrightarrow \Z[M]\longrightarrow \Z[M/ I]\longrightarrow 0 ,
	\]
	we obtain
	\[
	0=\Tor_{n+1}^{\Z[M]}(A,\Z[M])\longrightarrow \Tor_{n+1}^{\Z[M]}(A,\Z[M/I])\longrightarrow \Tor_n^{\Z[M]}(A,\Z[ I])=0
	\]
	and so $\Tor_{n+1}^{\Z[M]}(A,\Z[M/ I])=0$.
\end{proof}

\begin{theorem}\label{6.09}
	Let $((A,\mmm,k),M)$ be a noetherian local prelog ring and $\mmm_M=M-M^*$. The following are equivalent:
	\begin{enumerate}
		\item[(i)] $(A,M)$ is a log regular local ring.
		\item[(ii)] The map
		\[ H_2(A,k,k)\oplus \Tor_1^{\Z}(\ker(M^\gp \to (M/\mmm_M)^\gp),k) \longrightarrow H_2((A,M),(k,M/\mmm_M),k) \]
		is zero on the first summand and an isomorphism on the second one.
	\end{enumerate}
	If moreover $M$ is integral, then these conditions imply:
	\begin{enumerate}
		\item[(iii)] $H_n((A,M),(k,M/\mmm_M),k)=0$ for any $n\geq 3$.
	\end{enumerate}
\end{theorem}
\begin{proof}
	Equivalence of (i) and (ii) follows from Theorem~\ref{6.01}. We will see (iii). By Corollary~\ref{6.08}, $\Tor_n^{\Z[M]}(A,\Z[M/\mmm_M])=0$ for all $n>0$ and so by base change \cite[4.54]{An-1974} we have isomorphisms
	\[ H_i(\Z[M],\Z[M/\mmm_M],k) \longrightarrow H_i(A,A/\alpha(\mmm_M)A,k)\]
	for all $i\geq 0$. Since $A/\alpha(\mmm_M)A$ is regular, we have isomorphisms \cite[6.26, 5.1]{An-1974}
	\[ H_i(A,A/\alpha(\mmm_M)A,k) \longrightarrow H_i(A,k,k)\]
	for all $i\geq 2$. Therefore
	\[ H_i(\Z[M],\Z[M/\mmm_M],k) \longrightarrow H_i(A,k,k) \]
	are isomorphisms for all $i\geq 2$, and then from the fundamental exact sequence we deduce
	\[ H_n((A,M),(k,M/\mmm_M),k)=0 \]
	for all $n\geq 3$. 
\end{proof}

\begin{definition}\label{6.10}
	We say that a local prelog ring $((A,\mmm,k),M)$ is \emph{log complete intersection} if\linebreak
	$H_3((A,M),(k,M/\mmm_M),k)=0$. In particular, a log regular ring with $M$ integral is log complete intersection.
\end{definition}

From the fundamental exact sequence we obtain the following theorem.

\begin{theorem}\label{6.11}
	Let $((A,\mmm,k),M)$ be a noetherian local prelog ring with $M$ an integral monoid.
	\begin{enumerate}
		\item[(i)] If $(A,M)$ is log regular, then $A$ is a regular local ring if and only if $\ker(\Z[M]\to\Z[N])$ is generated by a regular sequence.
		\item[(ii)] If $A$ is a regular local ring, then $(A,M)$ is log complete intersection if and only if $\ker(\Z[M]\to\Z[M/\mmm_M])$ is generated by a regular sequence.
	\end{enumerate}
\end{theorem}
\begin{proof}
	It follows from Theorem~\ref{4.03}, Theorem~\ref{6.09} and \cite[6.25]{An-1974}, having in mind that we have seen at the end of the proof of Theorem~\ref{6.01} that the map \[ H_2(\Z[M],\Z[M/\mmm_M],k)\longrightarrow \Tor_1^{\Z[M]}(\ker(M^\gp\to(M/\mmm_M)^\gp),k) \] vanishes.
\end{proof}

\begin{theorem}\label{6.12}
	Let $(A,M)$ be a noetherian local prelog ring, $J$ a proper ideal of $A$, $B=A/ J$, $N=M/\alpha_M^{-1}( J)$ where $\alpha_M:M\to A$ is the structural map.
	\begin{enumerate}
		\item[(i)] If $(A,M)$ is log regular and $(B,N)$ is log complete intersection, then $ J$ is a log regular ideal.
		\item[(ii)] If $(B,N)$ is log regular and $ J$ is a log regular ideal, then $(A,M)$ is a log regular local ring.
		\item[(iii)] Assume that $M$ is integral and $M^\gp$ free (for instance if $M$ is saturated and sharp \cite[I.1.3.5]{Og}). If $(A,M)$ is log regular (or log complete intersection) and $J$ is log regular, then $(B,N)$ is log complete intersection.
	\end{enumerate}
\end{theorem}
\begin{proof}
	(i) Assume first that $\alpha_M^{-1}(J)\neq\emptyset$, so $N^\gp=\{1\}$.
	
	If $(B,N)$ is log complete intersection, then the map
	\[ \alpha\colon H_2((A,M),(B,N),k) \longrightarrow H_2((A,M),(k,M/\mmm_M),k)\]
	is injective. In the commutative diagram
	\[
	\begin{tikzcd}[row sep=4em, column sep=4em]
	H_2(\Z[M],\Z[N],k) \arrowr\arrowd & H_2(\Z[M],\Z[M/\mmm_M],k) \arrowd \\
	H_2(A,B,k)\oplus \Tor_1^{\Z}(M^\gp,k) \arrowr{\varepsilon\oplus\id}\arrowd{\beta_1\oplus\beta_2} & H_2(A,k,k)\oplus \Tor_1^{\Z}(M^\gp,k) \arrowd{\gamma_1\oplus\gamma_2} \\
	H_2((A,M),(B,N),k) \arrowr{\alpha} & H_2((A,M),(k,M/\mmm_M),k)
	\end{tikzcd}
	\]
	we know that $\gamma_2$ is an isomorphism and $\gamma_1=0$. Therefore $\beta_2$ is surjective and $\beta_1=0$. But $\beta_2$ is injective (we have seen in the last lines of the proof of Theorem~\ref{6.01} that the map
	$H_2(\Z[M],\Z[N],k) \longrightarrow \Tor_1^{\Z}(M^\gp,k)$
	vanishes), and then $J$ is log regular by Theorem~\ref{6.01}.
	
	Now if $\alpha_M^{-1}(J)= \emptyset$, then $M=N$, so $\ker(M^\gp\to N^\gp)=\{1\}$. Therefore $J$ is a log regular ideal if and only if $H_2((A,M),(B,N),k)=0$. As in the previous diagram, we obtain that the map
	\[ \beta_1\colon H_2(A,B,k) \longrightarrow H_2((A,M),(B,N),k) \]
	vanishes. But since $M=N$, this map is an isomorphism (Example \ref{4.02}).
	
	(ii) Assume first that $\alpha_M^{-1}(J)\neq \emptyset$. We have the above commutative diagram again, where now $\alpha$ is surjective (since $0= \Tor_1^{\Z}(N^\gp,k)=H_2((B,N),(k,M/\mmm_M),k)$), $\beta_1=0$ and $\beta_2$ is an isomorphism. This implies that $\gamma_2$ is an isomorphism since the map $H_2(\Z[M],\Z[M/\mmm_M],k)\to \Tor_1^{\Z}(M^\gp,k)$ is zero as above. Also, $\gamma_1=0$ by diagram chasing. If $\alpha_M^{-1}(J)=\emptyset$, the proof is similar to the proof of (i) when $\alpha_M^{-1}(J)\neq \emptyset$.
	
	(iii) By Corollary~\ref{6.04}, $J$ is log regular if and only if $H_2((A,M),(B,N),k)=0$. Then the result follows from the Jacobi-Zariski exact sequence and Theorem~\ref{6.09}.(iii).
\end{proof}

\begin{example}\label{6.13}
	Let $(A,\mmm,k),M)$ be a noetherian local prelog ring with $M$ integral saturated and sharp. Let $\hat{A}$ be the completion of $A$ for the $\mmm$-adic topology and consider the prelog ring $(\hat{A},M)$. We have $H_n((A,M),(\hat{A},M),k)=0$ for all $n\geq0$ by Example~\ref{4.02} and \cite[10.22]{An-1974}. Then from the Jacobi-Zariski exact sequence we obtain
	\[ H_n((A,M),(k,M/\mmm_M),k)=H_n((\hat{A},M),(k,M/\mmm_M),k) \]
	for all $n\geq0$. In particular, $(A,M)$ is log regular (resp. log complete intersection) if and only if so is $(\hat{A},M)$. Let $R:=C[[M]][[x_1,\ldots,x_n]]\to\hat{A}$ be a surjective homomorphism of rings with $C$ a Cohen ring or a field \cite[I.3.6]{Og}. Then $(R,M)$ is log regular (\cite[Theorem 3.2]{Ka-TS} or \cite[III.1.11.2]{Og}) and by Theorem~\ref{6.12} $(A,M)$ is log complete intersection if and only if $\ker(R\to\hat{A})$ is log regular. So $(A,M)$ is log complete intersection if and only if $(\hat{A},M)$ is the quotient of a log regular ring by a regular ideal. This fact, together with theorems~\ref{6.01} and \ref{6.09}, gives a description of log complete intersection rings.
\end{example}

Now we show how these homological methods allow us to compare log smoothness and log regularity. The following result is due to Kato \cite[8.3]{Ka-TS} in the finitely generated case.

\begin{theorem}\label{6.14}
	Let $((B,\nnn,l),N)$ be a noetherian local prelog ring with $N$ integral and saturated, and let $k\subset B$ be a field.
	\begin{enumerate}
		\item[(i)] If $(B,N)$ is a log formally smooth over $(k,1)$ for the $\nnn$-adic topology, then $(B,N)$ is log regular.
		\item[(ii)] Assume that $l|k$ is separable. If $(B,N)$ is log regular, then $(B,N)$ is a log formally smooth over $(k,1)$ for the $\nnn$-adic topology.
	\end{enumerate}
\end{theorem}
\begin{proof}
	By Theorem~\ref{6.03}, Proposition~\ref{5.04} and \cite[12.1.36]{GR} we can assume that $N$ is also sharp, and then that $N^\gp$ is free \cite[I.1.3.5]{Og}. Theorem~\ref{6.09} says then that $(B,N)$ is log regular if and only if $H_2((B,N),(l,N/\mmm_N),l)=0$, while Theorem~\ref{5.14} says that $(B,N)$ is a log formally smooth over $(k,1)$ if and only if $H_1((k,1),(B,N),l)=0$. Since we have a Jacobi-Zariski exact sequence
	\[
	\begin{tikzcd}[column sep=2.5em,row sep=0.8em]
	& & H_2((k,1),(l,N/\mmm_N),l) \arrowr & H_2((B,N),(l,N/\mmm_N),l) \arrowr & \ \\
	\arrowr & H_1((k,1),(B,N),l) \arrowr & H_1((k,1),(l,N/\mmm_N),l) \ ,  & 	\end{tikzcd}
	\]
	it suffices to show that $H_2((k,1),(l,N/\mmm_N),l)=0$, and that if $l|k$ is separable then $H_1((k,1),(l,N/\mmm_N),l)=0$.
	
	Since $N/\mmm_N=\{0,1\}$, we have a fundamental exact sequence
	\[
	\begin{tikzcd}[column sep=2.5em,row sep=0.8em]
	& & H_2(k,l,l) \arrowr & H_2((k,1),(l,N/\mmm_N),l) \arrowr & \ \\
	\arrowr & H_1(\Z,\Z[x]/(x-x^2),l) \arrowr & H_1(k,l,l) \arrowr & H_1((k,1),(l,N/\mmm_N),l) \\
	\arrowr & H_0(\Z,\Z[x]/(x-x^2),l)&  &
	\end{tikzcd}
	\]
	where $\varepsilon\colon \Z[x]\to l$ is the map sending $x$ to zero.
	
	By \cite[7.4, 7.11, 7.13]{An-1974} we have $H_2(k,l,l)=0$, and if $l|k$ is separable then $H_1(k,l,l)=0$. Therefore it suffices to show $H_i(\Z,\Z[x]/(x-x^2),l)=0$ for $i=0,1$. We have an exact sequence
	\[
	\begin{tikzcd}[column sep=2.5em,row sep=0.8em]
	& H_1(\Z,\Z[x],l) \arrowr & H_1(\Z,\Z[x]/(x-x^2),l) \arrowr & H_1(\Z[x],\Z[x]/(x-x^2),l) \arrowr{\varphi} & \ \\
	\arrowr{\varphi} & H_0(\Z,\Z[x],l) \arrowr & H_0(\Z,\Z[x]/(x-x^2),l) \arrowr & H_0(\Z[x],\Z[x]/(x-x^2),l)
	\end{tikzcd}
	\]
	which by \cite[3.36, 6.1, 6.3]{An-1974} takes the form
	\[
	\begin{tikzcd}[column sep=2.5em,row sep=0.8em]
	& 0 \arrowr & H_1(\Z,\Z[x]/(x-x^2),l) \arrowr & (x-x^2)\otimes_{\Z[x]}l \arrowr{\varphi} & \ \\
	\arrowr{\varphi} & \Omega_{\Z[x]|\Z}\otimes_{\Z[x]}l \arrowr & H_0(\Z,\Z[x]/(x-x^2),l) \arrowr & 0 \ .
	\end{tikzcd}
	\]
	Since
	\[ (x-x^2)\otimes 1 \overset{\varphi}{\longmapsto} (dx-2xdx)\otimes1=dx\otimes1-2dx\otimes\varepsilon(x)=dx\otimes1-2dx\otimes0=dx\otimes1 , \]
	we have that $\varphi$ is an isomorphism of rank 1 $l$-vector spaces, and so $H_i(\Z,\Z[x]/(x-x^2),l)=0$ for $i=0,1$.
\end{proof}

\begin{theorem}\label{6.15}
	Let $(B,N)$ be a noetherian local prelog ring with $N$ integral and saturated containing a perfect field $k$. If $(B,N)$ is log formally smooth over $(k,1)$ for the topology of the maximal ideal, then $\Omega_{(B,N)|(k,1)}$ is a flat $B$-module.
\end{theorem}
\begin{proof}
	$(B,N)$ is log regular by Theorem~\ref{6.14}, and by \cite[12.5.47]{GR} for any prime ideal $\qqq$ of $B$, $(B_\qqq,N_\qqq)$ is log regular, where $N_\qqq:=T^{-1}N$ with $T=N-\alpha_B^{-1}(\qqq)$. Since $N_\qqq$ is integral and saturated, again by Theorem~\ref{6.14}, $(B_\qqq,N_\qqq)$ is log formally smooth over $(k,1)$ for the $\qqq B_\qqq$-adic topology. By Proposition~\ref{4.10} and Theorem~\ref{5.14} we have
	\[ H_1((k,1),(B,N),W) \,=\, H_1((k,1),(B_\qqq,N_\qqq),W)\,=\,0 \]
	for any $k(\qqq)$-module $W$, where $k(\qqq)$ is the residue field of $B_\qqq$. 
	
	By \cite[12.1.36]{GR} there exits $N'\to N_\qqq$, with $N'$ integral saturated and sharp, inducing the same log structure in $B$. Therefore by propositions~\ref{4.10} and \ref{4.14} we have isomorphisms
	\[ H_{n}((k,1),(B,N),W) \,=\, H_{n}((k,1),(B_\qqq,N_\qqq),W) \,=\, H_{n}((k,1),(B_\qqq,N'),W)\]
	for any $n\geq0$ and any $k(\qqq)$-module $W$.
		
	Now consider the Jacobi-Zariski exact sequence
	\[
	\begin{tikzcd}[column sep=2.5em,row sep=0.8em]
	& H_{n+1}((k,1),(k(\qqq),N'/\mmm_{N'}),W) \arrowr & H_{n+1}((B_\qqq,N'),(k(\qqq),N'/\mmm_{N'}),W) \arrowr & \ \\
	\arrowr & H_n((k,1),(B_\qqq,N'),W) \arrowr & H_{n}((k,1),(k(\qqq),N'/\mmm_{N'}),W)  &
	\end{tikzcd}
	\]
	
	By Theorem~\ref{6.09} we have $H_{n+1}((B_\qqq,N'),(k(\qqq),N'/\mmm_{N'}),W)=0$ for all $n\geq 2$ and any $k(\qqq)$-module $W$. Since $N'/\mmm_{N'}=\{0,1\}$, we also have $H_{n}((k,1),(k(\qqq),N'/\mmm_{N'}),W)=0$ for all $n\geq1$ as we saw in the proof of Theorem~\ref{6.14}. Therefore $H_n((k,1),(B_\qqq,N'),W)=0$ for all $n\geq2$ and any $k(\qqq)$-module $W$.
	
	We have proved
	\[ H_n((k,1),(B,N),W)=0\]
	for all $n\geq1$, any $k(\qqq)$-module $W$ and any prime ideal $\qqq$ of $B$. Now the same proof as in \cite[Suppl\'ement, Proposition 29]{An-1974} gives
	\[ H_n((k,1),(B,N),X)=0\]
	for all $n\geq1$ and any $B$-module $X$.
	
	Finally, the universal coefficient exact sequence
	\[ H_1((k,1),(B,N),B)\otimes_BX \longrightarrow H_1((k,1),(B,N),X) \longrightarrow \Tor_1^B(H_0((k,1),(B,N),B),X)\longrightarrow 0 \]
	gives us that
	$\Omega_{(B,N)|(k,1)}=H_0((k,1),(B,N),B)$ is a flat $B$-module.
\end{proof}

\end{document}